\documentclass[12pt,reqno]{amsart}
\usepackage{graphicx}
\usepackage{amssymb,amsmath}
\usepackage{amsthm}
\usepackage{color}
\usepackage[pdf]{pstricks}
\usepackage{hyperref}
\usepackage{empheq}
\usepackage{pgfplots}

\usepackage{amssymb}
\usepackage{epsfig}
\usepackage{extarrows}
\usepackage{amsfonts}
\usepackage{graphicx}
\usepackage{caption}
\usepackage{subcaption}
\usepackage{tikz}

\newcommand{\R}{\mathbb{R}}

\newcommand{\N}{\mathbb{N}}

\newcommand{\eps}{\varepsilon}

\newtheorem{remark}{Remark}

\newtheorem{theorem}{Theorem}
\newtheorem{corollary}{Corollary}
\newtheorem{lemma}{Lemma}

\begin{document}
	
\title[Long-wave approximation in cylindrical coordinates]{\bf On the long-wave approximation of solitary waves in cylindrical coordinates}

\author{James Hornick}
\address[J. Hornick]{Department of Mathematics and Statistics, McMaster University, Hamilton, Ontario, Canada, L8S 4K1}
	
\author{Dmitry E. Pelinovsky}
\address[D. E. Pelinovsky]{Department of Mathematics and Statistics, McMaster University, Hamilton, Ontario, Canada, L8S 4K1}
\email{dmpeli@math.mcmaster.ca}
	
\author{Guido Schneider}
\address[G. Schneider]{Institut f\"{u}r Analysis, Dynamik und Modellierung,
			Universit\"{a}t Stuttgart, Pfaffenwaldring 57, D-70569 Stuttgart, Germany }

\begin{abstract}
	We address justification and solitary wave solutions of the cylindrical KdV equation which is formally derived as a long wave approximation of radially symmetric waves in a two-dimensional nonlinear dispersive system. For a regularized Boussinesq equation, we prove error estimates between true solutions of this equation and the associated cylindrical KdV approximation in the $L^2$-based spaces. The justification result holds in the spatial dynamics formulation of the regularized Boussinesq equation. We also prove that the class of solitary wave solutions considered previously in the literature does not contain solutions in the $L^2$-based spaces. This presents a serious obstacle in the applicability of the cylindrical KdV equation for modeling of radially symmetric solitary waves since the long wave approximation has to be performed separately in different space-time regions.
\end{abstract}

\maketitle

\section{Introduction} 

Long radially symmetric waves in a two-dimensional nonlinear dispersive system can be modeled with the cylindrical Korteweg-de Vries (cKdV) equation. 
The cKdV equation has been derived in \cite{Iordansky,Johnson80,Maxon,Miles} by perturbation theory from the equations of the water wave problem in cylindrical coordinates to describe radially symmetric waves going to infinity.
See \cite{Horikis,Johnson03} for an overview about the occurrence of this and other amplitude equations for the shallow water wave problem.

Derivation of the cKdV equation is not straightforward compared to its analog in rectangular coordinates, the classical KdV equation, and it is still an active  area of research in physics \cite{Grimshaw19,Karima1,Karima2,Karima3}.  No mathematically rigorous results have been derived for the justification of the cKdV equation, compared to the rigorous approximation results available for the classical KdV equation after the pioneering works   \cite{Craig,KanoNishida,SchnICIAM,SchneiderWayne}.
{\em The main objective of this paper is to prove an approximation result 
for the cKdV equation and to discuss the validity of this approximation. }

Although we believe that our methods can be applied to every nonlinear dispersive wave system where the cKdV equation can be formally derived we restrict ourselves in the following to the system given by a regularized Boussinesq equation. The regularized Boussinesq equation in two spatial dimensions can be written in the normalized form as 
\begin{equation}
\label{constintro}
\partial_t^2 u - \Delta u + \sigma \partial_t^2 \Delta u 
= \Delta (u^2) ,
\end{equation}
with space variable $ x = (x_1,x_2) \in \R^2 $, time variable $ t \in \R $, 
Laplacian $ \Delta = \partial_{x_1}^2 + \partial_{x_2}^2 $,
and a smooth solution $ u = u(x,t) \in \R $. The normalized parameter $\sigma = \pm 1$ determines the dispersion relation of linear waves $u \sim e^{i k \cdot x - i \omega t}$ for $k = (k_1,k_2) \in \R^2$ in the form:
\begin{equation}
\label{dispersion}
\omega^2 = \frac{|k|^2}{1 - \sigma |k|^2}, \qquad k \in \R^2.
\end{equation}
It follows from the dispersion relation (\ref{dispersion}) and the standard analysis of well-posedness \cite{Bona1,Bona2} that the initial-value problem for (\ref{constintro}) with the initial data
\begin{equation}
\label{ivp-time}
u|_{t = 0} = u_0(x), \qquad \partial_{t} u |_{t = 0} = u_1(x), \qquad x \in \R^2,
\end{equation}
is locally well-posed in Sobolev spaces of sufficient regularity for $\sigma = -1$ and ill-posed for $\sigma = +1$.

\begin{remark}
To justify the cKdV equation, we shall use the spatial dynamics formulation with the radius $ r := \sqrt{x_1^2 + x_2^2} $ as evolutionary variable. It turns out that due to the dispersion term $\sigma \partial_t^2 \Delta u$ in (\ref{constintro}) the spatial dynamics formulation and the temporal dynamics formulation are not well posed simultaneously.	If the temporal dynamics formulation is well posed, the spatial dynamics formulation is ill posed and vice versa.
\end{remark}
 
The radial spatial dynamics formulation of the regularized Boussinesq equation (\ref{constintro}) is obtained by introducing 
the radial variable $ r = \sqrt{x_1^2+x_2^2} $ and rewriting  \eqref{constintro} for $u = u(r,t)$ with $ \Delta = \partial_r^2 + \frac{1}{r}\partial_r $ in the form:
\begin{equation}
\label{Bous-rad}
(\partial_r^2 + r^{-1} \partial_r) (u - \sigma \partial_t^2 u + u^2) = 
\partial_t^2 u .
\end{equation}
The associated spatial dynamics problem is given by 
\begin{equation}
\label{ivp-radial}
u|_{r = r_0} = u_0(t), \qquad \partial_{r} u |_{r = r_0} = u_1(t), \qquad t \in \R,
\end{equation}
for some $ r_0 > 0 $. It is clear that the spatial evolution of  (\ref{Bous-rad}) with ``initial data" in (\ref{ivp-radial})  is locally well-posed for $\sigma = 1$ and ill-posed for $\sigma = -1$, see Theorem \ref{thexist}. In Section \ref{newsec2} we derive the cKdV equation for long waves of the radial Boussinesq equation (\ref{Bous-rad}) in case $ \sigma = 1 $. 
The cKdV approximation is given by $u(r,t) = \eps^2 A (\eps^3 r, \eps (t-r))$ 
with $\eps$ being a small parameter and $A(\rho,\tau)$ satisfying the following cKdV equation
\begin{equation}
\label{cKdV-intro}
2 \partial_{\rho} A + \rho^{-1} A + \partial_{\tau}^3 A = \partial_{\tau} (A^2), 
\end{equation}
where $ \tau := \eps( t - r ) \in \mathbb{R} $ and $ \rho :=\eps^3 r \geq \rho_0 $ for some $\rho_0 > 0$  are rescaled versions of the variables $(r,t)$ in the traveling frame and $A(\rho,\tau) \in \mathbb{R}$ is the small-amplitude approximation for $u(r,t) \in \mathbb{R}$. We have to impose the spatial dynamics formulation for the cKdV equation (\ref{cKdV-intro}) with the initial data 
\begin{equation}
\label{cKdV-IV}
A|_{\rho = \rho_0} = A_0(\tau), \quad \tau \in \mathbb{R}.
\end{equation}
It follows from the contraction mapping principle applied to the KdV equation \cite{KPV93} and the boundedness of the linear term $\rho^{-1} A$ for $\rho \geq \rho_0 > 0$ that the initial-value problem for (\ref{cKdV-intro}) with ``initial data" in (\ref{cKdV-IV}) is locally well-posed for $A_0 \in H^s(\mathbb{R})$ with any $s > \frac{3}{4}$. Moreover, if $\int_{\mathbb{R}} A_0 d \tau = 0$, then 
\begin{equation}
\label{mean-value}
\int_{\mathbb{R}} A(\rho,\tau) d \tau = 0, \qquad \rho \geq \rho_0,
\end{equation}
which implies that the unique local solution of (\ref{cKdV-intro})--(\ref{cKdV-IV}) satisfies 
\begin{equation}
\label{class-solutions}
A \in C^0([\rho_0,\rho_1],H^{s}(\R)) \quad \mbox{\rm and} 
\quad \partial_{\tau}^{-1} A \in C^0([\rho_0,\rho_1],H^{s-2}(\R))
\end{equation}
for some $\rho_1 > \rho_0$ if $A_0 \in H^s(\R)$  and $\partial_{\tau}^{-1} A \in H^{s-2}(\R)$, see Lemma \ref{lemma-Bous}.

The main approximation result is given by the following theorem. 

\begin{theorem}
	\label{mainthapp}
	Fix $ s_A > \frac{17}{2} $,  $ \rho_1 > \rho_0 > 0 $,  and $ C_{1} > 0 $.   Then there exist $ \varepsilon_0 > 0 $ and $ C_0 > 0 $  such that 
	for all $ \varepsilon \in (0,\varepsilon_0) $
	the following holds. 
	Let $ A \in C^0([\rho_0,\rho_1],H^{s_A}(\mathbb{R})) $  be a  solution of the cKdV equation \eqref{cKdV-intro}
	with 
	$$ 
	\sup_{\rho \in [\rho_0,\rho_1]}(\| A(\rho,\cdot) \|_{H^{s_A}} + \|  \partial_{\tau}^{-1}  A(\rho,\cdot) \|_{H^{s_A-2}} ) \leq C_1 .
	$$ 
	Then there are solutions $ (u, \partial_r u) \in C^0([\rho_0 \varepsilon^{-3}, \rho_1 \varepsilon^{-3}],H^s(\R) \times H^s(\R)) $ of the radial Boussinesq equation \eqref{Bous-rad} with $s > \frac{1}{2}$ satisfying 
	$$ 
	\sup_{r \in [\rho_0 \varepsilon^{-3}, \rho_1 \varepsilon^{-3}]} 
	\sup_{t \in \R}| u(r,t) - \varepsilon^2 A(\varepsilon^3 r, \varepsilon(t-r)) | \leq C_0 \varepsilon^{7/2} .
	$$ 
\end{theorem}

\begin{remark}
The proof of Theorem \ref{mainthapp}  goes along the lines of the associated proof for validity of the KdV approximation in \cite{SchnICIAM,SchneiderWayne}. However, there are new difficulties which have to be overcome. The major point is 
that a vanishing mean value as in (\ref{mean-value}) is required for the solutions of the cKdV equation (\ref{cKdV-intro}), a property which  fortunately is preserved by the evolution of the cKdV equation.  Subsequently, a vanishing  mean value is also required for the solutions of the radial Boussinesq equation (\ref{Bous-rad}).  However, this property is not preserved in the spatial evolution of the radial Boussinesq equation (\ref{Bous-rad}). We use a nonlinear change of variables from $u(r,t)$ to $v(r,t)$ in Section \ref{newsec2} in order to preserve the vanishing mean value in the spatial evolution. 
\end{remark}

The cKdV equation (\ref{cKdV-intro}) admits exact solutions for solitary waves due to its integrability \cite{Calogero,Dryuma,NakamuraChen80}. These exact solutions have important physical applications \cite{Johnson90,Hersh,Step81,Weidman1,Weidman2},
which have continued to stimulate recent research \cite{Gal,Step1,Step2}. 
It was observed that parameters of the exact solutions of the cKdV equation agree well with the experimental and numerical simulations of solitary waves. 
However, the solitary wave solutions of the cKdV equation do not decay sufficiently well at infinity \cite{Johnson99} and hence it is questionable how such solutions can be described in the radial spatial dynamics 
of the Boussinesq equation in the mathematically rigorous sense.

We address the solitary wave solutions of the cKdV equation (\ref{cKdV-intro}) in Section \ref{newsec3}, where we will use the theory of Airy functions and give a more complete characterization of the solitary wave solutions compared to previous similar results, e.g. in Appendix A of \cite{Johnson99}. The following theorem presents the corresponding result.

\begin{theorem}
	\label{th-soliton}
Consider solutions of the cKdV equation (\ref{cKdV-intro}) in the class of solitary waves given by 
\begin{equation}
\label{A-Hirota}
	A(\rho,\tau) = -6 \partial_{\tau}^2 \log \left[ 1 + \frac{1}{(6\rho)^{1/3}} F\left( \frac{\tau}{(6 \rho)^{1/3}} \right) \right],
\end{equation}
with some $F \in C^{\infty}(\mathbb{R},\R)$. All bounded solutions in the form \eqref{A-Hirota}  satisfy the decay condition 
$$
A(\rho,\tau) \to 0 \quad \mbox{\rm as } \;\; |\tau| \to \infty
$$ 
and the zero-mean constraint 
$$
\int_{\mathbb{R}} A(\rho,\tau) d \tau = 0
$$
but fail to be square integrable, that is, $A(\rho,\cdot) \notin L^2(\R)$ for every $\rho > 0$. 
\end{theorem}

\begin{remark}
	The result of Theorem \ref{th-soliton} is due to the slow decay of solitary wave solutions (\ref{A-Hirota}) with 
	$$
	A(\rho,\tau) \sim |\tau|^{-1/2} \quad \mbox{\rm as} \quad \tau \to -\infty.
	$$ 
	We note that such solitary wave solutions satisfy 
	$A \in C^0((0,\infty),\dot{H}^s(\R))$ for any $s > 0$ but we are not aware of the local well-posedness for the cKdV equation (\ref{cKdV-intro}) in $\dot{H}^s(\R)$ with $s > 0$. Consequently, the justification result of Theorem \ref{mainthapp} does not apply to the solitary waves of the cKdV equation (\ref{cKdV-intro}) and one needs to use matching techniques in different space-time regions in order to consider radial solitary waves diverging from the origin, cf. \cite{Johnson80,Johnson90,Johnson99}.
\end{remark}

Similar questions arise for the long azimuthal perturbations of  the long radial waves. A cylindrical Kadomtsev-Petviashvili (cKP) equation was also proposed as a relevant model in \cite{Johnson80,Johnson03}. Motivated from physics of fluids and plasmas, problems of transverse stability of ring solitons were studied recently in \cite{Kresh,Step2,Step3}. Other applications of the KP approximation are interesting in the context of dynamics of square two-dimensional lattices based on the models of the Fermi-Pasta-Ulam type \cite{GP21,HP22,PS23}. Radially propagating waves with azimuthal perturbations are natural objects in lattices, see, e.g., \cite{Mac23,ST18}, and clarification of the justification of 
the cKdV equation is a natural first step before justification of the cKP equation in nonlinear two-dimensional lattices. We discuss further implication of the results of Theorems \ref{mainthapp} and \ref{th-soliton} for the cKdV and cKP equations in Section \ref{newsec4}.

\medskip

{\bf Notation.} Throughout this paper different constants are denoted with the same symbol $ C $ if they can be chosen independently of the small parameter
$ 0 < \varepsilon \ll 1 $. The Sobolev space $ H^s(\R) $, $s \in \N$ of  $ s $-times weakly differentiable functions is equipped with the norm 
$$ 
\| u \|_{H^s} = \left( \sum_{j= 0}^{s} \int_{\R} |\partial_x^j u(x)|^2 \mathrm{d}x\right)^{1/2} .
$$ 
The weighted Lebesgue space $ L^2_s(\R) $, $s \in \R$  is equipped with the norm 
$$ 
\| \widehat{u} \|_{L^2_s} = \left( \int_{\R} |\widehat{u}(k)|^2 (1+k^2)^s \mathrm{d}k\right)^{1/2} .
$$
Fourier transform is an isomorphism between $ H^s(\R) $ and $ L^2_s(\R) $ which allows us to extend the definition of $H^s(\R)$ to all values of $s \in \R$. 
\medskip

{\bf Acknowledgement.} The work of G. Schneider was partially  supported by the Deutsche Forschungsgemeinschaft (DFG, German
Research Foundation) - Project-ID 258734477 - SFB 1173. 
D. E. Pelinovsky acknowledges the funding of this study provided by the grant No. FSWE-2023-0004 
and grant No. NSH-70.2022.1.5. 

\section{Justification of the cKdV equation}
\label{newsec2}

Here we prove Theorem \ref{mainthapp} which states the approximation result for the cKdV equation. The plan is as follows.
In Section \ref{sec2} we derive the cKdV equation  (\ref{cKdV-intro}) for the radial Boussinesq equation \eqref{Bous-rad} in case $ \sigma = 1 $.
In Section \ref{sec3} we estimate the residual terms, i.e., the terms which remain after inserting the cKdV approximation into the radial Boussinesq equation. In Section \ref{sec8} we prove a local existence and uniqueness result for the radial spatial dynamics formulation. In Section \ref{sec4}--\ref{sec4b} we estimate the error made by this formal approximation in the radial spatial dynamics by establishing $ L^2 $- and $ H^1 $-energy estimates. The argument is completed in Section \ref{sec6} by using the energy to control the approximation error and by applying Gronwall's inequality.

\subsection{Derivation of the cKdV equation}
\label{sec2}

We rewrite the radial Boussinesq equation (\ref{Bous-rad}) with $ \sigma = 1 $ as 
\begin{equation} 
\label{const1}
(\partial_r^2 + r^{-1} \partial_r) (u - \partial_t^2 u + u^2) = 
\partial_t^2 u .
\end{equation}
The cKdV approximation can be derived if $ r \geq r_0 > 0 $ is considered as the evolutionary variable with the initial data (\ref{ivp-radial}). 
However, this evolutionary system has the disadvantage that 
$ \int_{\mathbb{R}} u(r,t) dt $ is not preserved in $r$,  see Remarks \ref{rem2neu} and \ref{rem33}. In order to overpass this technical difficulty, we rewrite \eqref{const1} as 
\begin{equation} 
\label{const3}
\partial_{t}^2 ( 1 + \partial_r^2 + r^{-1} \partial_r)  u = (\partial_r^2 + r^{-1} \partial_r) (u + u^2) 
\end{equation}
and make the change of variables $ v := u + u^2 $. For small $ v $ this 
quadratic equation admits a unique solution for small $u$ given by 
$$ 
u = v - v^2 + N(v) 
$$ 
with analytic $ N(v) = \mathcal{O}(v^3) $. In variable $v$, the radial spatial evolution problem is 
\begin{equation} 
\label{const5}
(\partial_r^2 + r^{-1} \partial_r) v = \partial_{t}^2 ( 1 + \partial_r^2 + r^{-1} \partial_r)  (v - v^2 + N(v)) .
\end{equation}
The local existence and uniqueness of solutions of the initial-value problem 
\begin{equation*}
\left\{ \begin{array}{l} 
(\partial_r^2 + r^{-1} \partial_r) v = \partial_{t}^2 ( 1 + \partial_r^2 + r^{-1} \partial_r)  (v - v^2 + N(v)), \quad r > r_0, \\
v |_{r = r_0} = v_0, \quad \partial_r v |_{r = r_0} = v_1 \end{array} \right.
\end{equation*}
can be shown for $ (v_0,v_1) \in H^s(\R) \times H^s(\R) $ for every $ s > \frac{1}{2} $,  see Theorem \ref{thexist}.

We make the usual ansatz for the derivation of the KdV equation, namely 
\begin{equation} \label{2kdvansatz}
\varepsilon^2 \psi_{\rm cKdV}(r,t) := \eps^2 A(\eps^3 r,\eps( t - c r )) ,
\end{equation}
with $ \tau := \eps( t - c r ) $ and $ \rho := \eps^3 r $, where $c$ is the wave speed. Defining the residual
\begin{equation} \label{const4}
\textrm{Res}(v) := - (\partial_r^2 + r^{-1} \partial_r) v+ \partial_{t}^2 ( 1 + \partial_r^2 + r^{-1} \partial_r)  (v - v^2 + N(v)) 
\end{equation}
we find
\begin{align*}
\textrm{Res}(\varepsilon^2 \psi_{\rm cKdV}) & = - c^2 \eps^4 \partial_{\tau}^2 A + 2 c \eps^6 \partial_{\rho} \partial_{\tau} A - \eps^8 \partial_{\rho}^2 A \\ 
& \quad + c \eps^6 \rho^{-1} \partial_{\tau} A - \eps^8 \rho^{-1} \partial_{\rho} A \\
& \quad+ \eps^4 \partial_{\tau}^2 A + c^2 \eps^6 \partial_{\tau}^4 A - 2 c \eps^8 \partial_{\rho} \partial_{\tau}^3 A + \eps^{10} \partial_{\rho}^2 \partial_{\tau}^2 A \\ 
& \quad- c \eps^8 \rho^{-1}\partial_{\tau}^3 A + \eps^{10} \rho^{-1}\partial_{\rho} \partial_{\tau}^2 A 
\\ & \quad - \eps^6 \partial_{\tau}^2 (A^2) -  c^2 \eps^8 \partial_{\tau}^4 (A^2)+ 2 c \eps^{10} \partial_{\rho} \partial_{\tau}^3  (A^2) - \eps^{12} \partial_{\tau}^2 \partial_{\rho}^2  (A^2) \\ 
& \quad+ c \eps^{10} \rho^{-1}\partial_{\tau}^3  (A^2) - \eps^{12} \rho^{-1}\partial_{\rho} \partial_{\tau}^2 (A^2) \\ 
& \quad
+ \varepsilon^2 \partial_{\tau}^2 ( 1 + (- c \varepsilon \partial_{\tau} + \varepsilon^3 \partial_{\rho})^2 +\varepsilon^3  \rho^{-1} (- c \varepsilon \partial_{\tau} + \varepsilon^3 \partial_{\rho})) N(\varepsilon^2 A)
\end{align*}
where the last line is at least of order $ \mathcal{O}(\varepsilon^8) $. 
We eliminate the terms of $ \mathcal{O}(\varepsilon^4) $ by choosing $ c^2 = 1 $. The radial waves diverge from the origin if $c = 1$ and converge towards the origin if $c = -1$. It makes sense to consider only outgoing radial waves, so that we set $c = 1$ in the following. 

With $c = 1$, the terms of $ \mathcal{O}(\varepsilon^6) $ are eliminated in $\textrm{Res}(\varepsilon^2 \psi_{\rm cKdV})$ by choosing $ A $ to satisfy the cKdV equation (\ref{cKdV-intro}) rewritten here as 
\begin{equation} 
\label{cKdV}
2 \partial_{\rho}  A  + \rho^{-1} A + \partial_{\tau}^3 A
- \partial_{\tau} (A^2) = 0 .
\end{equation}
By this choice we formally have
$$ 
\textrm{Res}(\varepsilon^2 \psi_{\rm cKdV}) =  \mathcal{O}(\varepsilon^8) .
$$ 
We will estimate the residual terms rigorously in Section \ref{sec3}.

\begin{remark} \label{rem2neu}
In our subsequent  error estimates $ \partial_t^{-1} $ has to be applied 
to $\textrm{Res}(v)$ in (\ref{const4}). However,  this is only possible if the nonlinear change of variables $ v = u + u^2 $ is applied.  This change of variables also allows us to  use the variable $ \partial_t^{-1} \partial_r u $ which played 
a fundamental role in the justification of the KdV equation in \cite{SchnICIAM,SchneiderWayne} and which is necessary to obtain 
an $ L^2 $-bound for the approximation error.
\end{remark}

\subsection{Estimates for the residual}
\label{sec3}

For estimating the residual $\textrm{Res}(\varepsilon^2 \psi_{\rm cKdV})$ we consider a solution $ A \in C([\rho_0,\rho_1],H^{s_A}(\R,\R)) $
of the cKdV equation \eqref{cKdV} with some $ s_A \geq 0 $ suitably chosen below.
Let 
\begin{equation}
\label{assAsolu}
C_A := \sup_{\rho \in [\rho_0,\rho_1]} \|A(\rho,\cdot) \|_{H^{s_A}} < \infty.
\end{equation}
With $c = 1$ and $A$ satisfying the cKdV equation (\ref{cKdV}), the residual is rewritten as 
\begin{align*}
\textrm{Res}(\varepsilon^2 \psi_{\rm cKdV}) & = - \eps^8 \partial_{\rho}^2 A - \eps^8 \rho^{-1} \partial_{\rho} A - 2 \eps^8 \partial_{\rho} \partial_{\tau}^3 A + \eps^{10} \partial_{\rho}^2 \partial_{\tau}^2 A \\ 
& \quad - \eps^8 \rho^{-1}\partial_{\tau}^3 A + \eps^{10} \rho^{-1}\partial_{\rho} \partial_{\tau}^2 A -  \eps^8 \partial_{\tau}^4 (A^2) +  2 \eps^{10} \partial_{\rho} \partial_{\tau}^3  (A^2)  \\ 
& \quad - \eps^{12} \partial_{\tau}^2 \partial_{\rho}^2  (A^2) + \eps^{10} \rho^{-1}\partial_{\tau}^3  (A^2) - \eps^{12} \rho^{-1}\partial_{\rho} \partial_{\tau}^2 (A^2) \\ 
& \quad
+ \varepsilon^2 \partial_{\tau}^2 ( 1 + (- \varepsilon \partial_{\tau} + \varepsilon^3 \partial_{\rho})^2 +\varepsilon^3  \rho^{-1} (- \varepsilon \partial_{\tau} + \varepsilon^3 \partial_{\rho})) N(\varepsilon^2 A)
\end{align*}
We can express $ \rho $-derivatives of $ A $ by $ \tau $-derivatives of $ A $ through the right-hand side of the cKdV equation \eqref{cKdV}. Hence for replacing one $ \rho $-derivative we need three 
$ \tau $-derivatives. In this way, the term $ \eps^{10} \partial_{\rho}^2 \partial_{\tau}^2 A$  loses most derivatives, namely eight   $ \tau $-derivatives. Due to the scaling properties of the $ L^2 $-norm w.r.t. the scaling $ \tau = \varepsilon (t-r) $, we are loosing $ \eps^{-1/2} $ 
in the estimates, e.g., see \cite{SU17}. As a result of the standard analysis, we obtain the following lemma.

\begin{lemma}
Let $ s \geq 0 $. Assume \eqref{assAsolu} with $ s_A = s + 8 $ and $C_A > 0$. There exists a $ C_{\rm res} > 0 $ such that for all $ \varepsilon \in (0,1] $ we have
$$
\sup_{r \in [\rho_0 \varepsilon^{-3},\rho_1 \varepsilon^{-3}]}  \| {\rm Res}(\varepsilon^2 \psi_{\rm cKdV})(r,\cdot) \|_{H^s} \leq C_{\rm res} \varepsilon^{\frac{15}{2}}.
$$
\label{lem-res}
\end{lemma}

In the subsequent error estimates we also need estimates for 
$ \partial_t^{-1} $  applied to $\textrm{Res}(\varepsilon^2 \psi_{\rm cKdV})$. The only terms in the residual which have no $ \partial_t $ in front are the ones collected in 
$$ 
 \eps^8  R_1 = - \eps^8 \partial_{\rho}^2 A  - \eps^8 \rho^{-1} \partial_{\rho} A .
$$ 
When $ \partial_{\rho} A $ is replaced by the right-hand side of the cKdV equation (\ref{cKdV}), we find
\begin{align*}
R_1 & =  \frac{1}{2} (\partial_{\rho} + \rho^{-1}) 
( \rho^{-1} A + \partial_{\tau}^3 A -  \partial_{\tau} A^2 )  \\
&=  \frac{1}{2} \partial_{\tau} (\partial_{\rho} + \rho^{-1}) 
(  \partial_{\tau}^2 A - A^2 ) + \frac{1}{2} \rho^{-1} \partial_{\rho} A \\
&=  \frac{1}{4} \partial_{\tau} (2 \partial_{\rho} + \rho^{-1}) 
(  \partial_{\tau}^2 A - A^2 ) - \frac{1}{4} \rho^{-2} A.
\end{align*}
Therefore, all terms in the residual can be written as derivatives in $\tau$ except of the term   $  - (4 \rho^2)^{-1} A $.
The operator $ \partial_{\tau}^{-1} $, respectively a multiplication with $ \frac{1}{ik} $ in the Fourier space, can be applied 
to $ - (4 \rho^2)^{-1} A $ only if $ A $  has a vanishing mean value and its Fourier transform decays as $ \mathcal{O}(|k|) $ for $ k \to 0 $. This is why we enforce 
the vanishing mean value as in (\ref{mean-value}) and consider 
solutions of the cKdV equation in the class of functions (\ref{class-solutions}). Such solutions are given by the following lemma.

\begin{lemma}
Fix $ s_A > \frac{3}{4} $, $ \rho_0 > 0 $, and pick $A_0 \in H^{s_A}(\R)$ such that 
$ \partial_{\tau}^{-1}  A_0  \in H^{s_A-2}(\R)$. There exist $ C > 0 $ and $ \rho_1 > \rho_0 $ such that the cKdV equation \eqref{cKdV} possesses 
a unique solution $ A \in C^0([\rho_0,\rho_1],H^{s_A}(\R)) $ with $A|_{\rho = \rho_0} = A_0$ satisfying 
$$ 
 \sup_{\rho \in [\rho_0,\rho_1]}(\| A(\rho,\cdot) \|_{H^{s_A}} + \|  \partial_{\tau}^{-1}  A(\rho,\cdot) \|_{H^{s_A-2}} ) \leq C.
$$ 
\label{lemma-Bous}
\end{lemma}

\begin{proof}
The cKdV equation \eqref{cKdV} possesses a unique solution 
$ A \in C^0([\rho_0,\rho_1],H^{s_A}(\R)) $ for $ s_A > \frac{3}{4} $,  see \cite{KPV93}.
To obtain the estimate on $ B := \partial_{\tau}^{-1}  A $, we rewrite 
\eqref{cKdV} in the form:
$$
2 \partial_{\rho}  B +  \rho^{-1} B + \partial_{\tau}^2 A
- A^2 = 0 .
$$
Since $B_0 := \partial_{\tau}^{-1}  A_0  \in H^{s_A-2}(\R)$, integration of this equation with $ A \in C^0([\rho_0,\rho_1],H^{s_A}(\R)) $ yields 
$ B \in C^0([\rho_0,\rho_1],H^{s_A-2}(\R)) $.
\end{proof}

For estimating the residual ${\rm Res}(\varepsilon^2 \psi_{\rm cKdV})$  we consider a solution $ A \in C^0([\rho_0,\rho_1],H^{s_A}(\R)) $
of the cKdV equation \eqref{cKdV} with 
\begin{equation}
\label{assABsolu}
C_{A,B} := \sup_{\rho \in [\rho_0,\rho_1]}(\| A(\rho,\cdot) \|_{H^{s_A}} + \|  \partial_{\tau}^{-1}  A(\rho,\cdot) \|_{H^{s_A-2}} ) < \infty,
\end{equation}
and with $ s_A > \frac{3}{4} $ being sufficiently large.
Due to the correspondence $ \partial_t^{-1} = \varepsilon^{-1} \partial_{\tau}^{-1} $ we have the following lemma.

\begin{lemma}
Let $ s \geq 0 $. Assume \eqref{assABsolu} with $ s_A = s + 8 $ and $C_{A,B} > 0$. There exists a $ C_{\rm res} > 0 $ such that 
for all $ \varepsilon \in (0,1] $ we have
$$
\sup_{r \in [\rho_0 \varepsilon^{-3},\rho_1 \varepsilon^{-3}]}  \| \partial_t^{-1} {\rm Res}(\varepsilon^2 \psi_{cKdV})(r,\cdot) \|_{H^s} \leq C_{\rm res} \varepsilon^{\frac{13}{2}}.
$$
\label{lem-res-anti}
\end{lemma}

\begin{remark}\label{rem33}
Without the transformation $ v = u + u^2 $ which converts (\ref{const3}) into (\ref{const5}), the terms in the residual ${\rm Res}(u)$ constructed similarly to (\ref{const4}) which have no $ \partial_t $ in front would be
$$ 
- \eps^8 \partial_{\rho}^2 A  - \eps^8 \rho^{-1}\partial_{\rho} A - \eps^{10} \partial_{\rho}^2  (A^2)  - \eps^{10} \rho^{-1}\partial_{\rho}  (A^2) .
$$ 
As above by replacing $ \partial_{\rho} A $ by the right-hand side of the cKdV equation (\ref{cKdV}) we gain derivatives in $\tau$. However, due to the $  \rho^{-1} A $ term in \eqref{cKdV} among other terms we would produce terms of the form $ \eps^8 \rho^{-2} A $ and $ \eps^{10}  \rho^{-2} A^2 $.
The operator $ \partial_{\tau}^{-1} $ can only be applied to these terms if $ A $ and $ A^2 $ have a vanishing mean value. However, $ A^2 $ can only have a vanishing mean value if $ A $ vanishes identically. Moreover, it doesn't help to consider  $ \partial_{\rho}^2 \partial_{\tau}^{-1} (A^2)  $ and $ \rho^{-1} \partial_{\rho} \partial_{\tau}^{-1} (A^2) $ directly since  the cKdV equation (\ref{cKdV}) does not preserve the $ L^2 $-norm of the solutions.
Therefore, the transformation $ v = u + u^2 $ is essential for our justification analysis.
\end{remark}

\subsection{Local existence and uniqueness}
\label{sec8}

Here we prove the local existence and uniqueness of the solutions of
the second-order evolution equation (\ref{const5}), which we rewrite as 
\begin{equation*}
(\partial_r^2 + r^{-1}\partial_r)(1- \partial_t^2 ) v = \partial_t^2 v + 
\partial_t^2 ( 1 + \partial_r^2 + r^{-1}\partial_r)  (- v^2 + N(v)) .
\end{equation*}
By using  $ \mathcal{B}^2 :=  \partial_t^2 (1- \partial_t^2 )^{-1} $, 
we rewrite the evolution problem in the form:
\begin{align}
(\partial_r^2 + r^{-1}\partial_r) v & = \mathcal{B}^2v + 
\mathcal{B}^2 ( 1 + \partial_r^2 + r^{-1}\partial_r)  (- v^2 + N(v)) 
\notag \\ 
& =  \mathcal{B}^2 v + 
\mathcal{B}^2  (- v^2 + N(v)) + r^{-1} \mathcal{B}^2 (-2v +N'(v))\partial_r v
\notag \\ 
& \quad + \mathcal{B}^2 \left[ (-2 v+N'(v))  \partial_r^2 v +(-2+N''(v))(\partial_r v)^2 \right] .
\label{second-order}
\end{align}
The operator $ \mathcal{B}^2 $ is bounded in Sobolev space $ H^s(\R) $ for every $s \in \mathbb{R}$. The second-order evolution equation (\ref{second-order}) can be rewritten as a first-order system by introducing  $ w := \partial_r v $ such that 
\begin{equation} 
\label{diffeq1}
\left\{ \begin{array}{l} 
\partial_r v = w, \\
\partial_r w  = f(v,w) , \end{array} \right.
\end{equation}
where
\begin{align*}
	f(v,w) & = - r^{-1} w + \left[ 1 - \mathcal{B}^2 (-2 v+N'(v)) \cdot \right]^{-1} \mathcal{B}^2 \left[ v - v^2 + N(v) + (-2+N''(v))w^2 \right] .
\end{align*}
Since $N(v) = \mathcal{O}(v^3)$ for small $v$, the right hand side of system \eqref{diffeq1} for sufficiently small $ v $ is  locally Lipschitz-continuous  in  $ H^s(\R) \times H^s(\R) $ for every $ s > \frac{1}{2} $ due to Sobolev's embedding theorem. The following 
local existence and uniqueness result holds due to the Picard-Lindel\"{o}f theorem.

\begin{theorem}\label{thexist}
Fix $ s > \frac{1}{2} $ and $ r_0 > 0 $. There exists a $ \delta_0 > 0 $ such that for all $ \delta \in (0,\delta_0) $ and $ (v_0,w_0) \in H^s(\R) \times H^s(\R) $   with $ \| v_0 \|_{H^s}  \leq \delta $, there exists $ r_1 > r_0 $ and a unique solution $ (v,w) \in C^0([r_0,r_1],H^s(\R) \times H^s(\R)) $ of system \eqref{diffeq1} with $ (v,w) |_{r = r_0} = (v_0,w_0) $.  
\end{theorem}

\begin{corollary}
	There exists a unique solution 
	$ (v,\partial_r v) \in C^0([r_0,r_1],H^s(\R) \times H^s(\R)) $ of the second-order evolution equation \eqref{const5} for the corresponding $ (v,\partial_r v) |_{r = r_0} = (v_0,w_0) $ .
\end{corollary}

\begin{remark}\label{rem43}
A combination of the local existence and uniqueness result of Theorem \ref{thexist} with the subsequent error estimates, used as a priori estimates, guarantees the existence and uniqueness of the solutions of equations for the error terms, see equation \eqref{erroreq}, as long as the error is estimated to be small.
\end{remark}

\subsection{The $ L^2 $-error estimates}
\label{sec4}

We introduce the error function $ R $ through 
the decomposition 
$$ 
v = \varepsilon^2 \psi_{\rm cKdV} + \varepsilon^{\beta} R 
$$ 
with $\psi_{\rm cKdV}(r,t) = A(\rho,\tau)$ and $ \beta :=  \frac{7}{2} $ to be obtained from the energy estimates, see Section \ref{sec6}. The error function $ R $ satisfies 
\begin{align} 
0 & = (\partial_r^2 + r^{-1}\partial_r) R - \partial_t^2 (1+\partial_r^2 + r^{-1}\partial_r) R  \notag \\ & \quad 
+ 2\varepsilon^2 \partial_t^2 (1+\partial_r^2 + r^{-1}\partial_r) ( A R) 
+ \varepsilon^{\beta} \partial_t^2 (1+\partial_r^2 + r^{-1}\partial_r) (R^2) 
\notag
\\
& \quad -\varepsilon^{-\beta} \partial_t^2 ( 1 + \partial_r^2 + r^{-1}\partial_r) (N(\varepsilon^2 A + \varepsilon^{\beta} R) - N(\varepsilon^2 A)) \notag
\\& \quad - \varepsilon^{-\beta} \textrm{Res}(\varepsilon^2 A). \label{erroreq}
\end{align}
Before we start to estimate the error we note that there is no problem with regularity of solutions of equation \eqref{erroreq} in the following sense.  Rewriting \eqref{erroreq}  as \eqref{second-order} and \eqref{diffeq1} in Section \ref{sec8} shows that if $ (R,\partial_r R) \in C^0([r_0,r_1],H^s(\R) \times H^s(\R)) $, then $ \partial_r^2 R(r,\cdot) $ has the same regularity. 
In particular, we have the estimate: 

\begin{lemma}
There exist constant $ C_{\rm l} $ and a smooth monotone function $ C_{\rm n} $ such that for all $ \varepsilon \in (0,1) $ we have
\begin{align} \label{dttResti}
 \| \partial_r^2  R(r,\cdot) \|_{L^2} & \leq \varepsilon^{\frac{15}{2}-\beta} C_{\rm res} +  C_{\rm l} (1 + \varepsilon^2 C_A)  (\|  R(r,\cdot) \|_{L^2} +  \| \partial_r  R(r,\cdot) \|_{L^2}) \\ & \quad + \varepsilon^{\beta} C_{\rm n}(\|  R(r,\cdot) \|_{L^{\infty}} +  \| \partial_r  R(r,\cdot) \|_{L^{\infty}})
 (\|  R(r,\cdot) \|_{L^2} +  \| \partial_r  R(r,\cdot) \|_{L^2}) , \notag
\end{align}
where $C_{\rm res}$ is defined in Lemma \ref{lem-res} and $C_A$ is defined in \eqref{assAsolu}.
\label{lem-estimate}
\end{lemma}

\begin{remark}
The difficulty in estimating the error $R$ comes from fact that the error equation \eqref{erroreq} contains the linear terms of order $ \mathcal{O}(\varepsilon^2) $ while we have to bound the error on the interval 
$[\varepsilon^{-3} r_0, \varepsilon^{-3}r_1]$ of length $ \mathcal{O}(\varepsilon^{-3}) $. We get rid of this 
mismatch of powers in  $ \varepsilon $  by writing the terms of order $ \mathcal{O}(\varepsilon^2) $ as derivatives in $r$ such that these  can be either included in the balance of energy or be written as terms 
where derivatives fall on $ A $ which allows us to estimate these  terms to be of  order $ \mathcal{O}(\varepsilon^3) $.
\end{remark}

We  follow the approach used in the energy estimates for the KdV approximation for obtaining an $ H^1 $-estimate for $ R $ \cite{SchnICIAM,SchneiderWayne}. To obtain first the $ L^2 $-estimates for $ R $, we multiply \eqref{erroreq} with $ - \partial_r \partial_t^{-2} R $ and  integrate it w.r.t. $ t $. 
The term $ - \partial_r \partial_t^{-2} R $ is defined via its Fourier transform w.r.t. $ t $, i.e., with abuse of notation, by  $  \partial_t^{-1} R  = \mathcal{F}^{-1}( (ik)^{-1} \widehat{R})$. All integrals in $t$ are considered on $\R$ and Parseval's equality is used when it is necessary.
We report details of computations as follows. \\

{\bf i)} From the linear terms in $ R $ we then obtain 
\begin{align*}
s_{1} & = -\int (\partial_r^2 R )(\partial_r \partial_t^{-2}  R) dt  = \frac12 \frac{d}{dr}  \int (\partial_r \partial_t^{-1} R)^2 d t, \\ 
s_{2} & = -\int ( r^{-1}\partial_r R  )(\partial_r \partial_t^{-2}  R) dt  =  r^{-1} \int
(\partial_r \partial_t^{-1} R)^2 dt , \\ 
s_{3} & =   \int (\partial_t^2 R )(\partial_r \partial_t^{-2}  R) dt  = \frac12 \frac{d}{dr}  \int
R^2 dt, 
\\ 
s_{4} & =  \int  (\partial_t^2 \partial_r^2 R )(\partial_r \partial_t^{-2}  R) dt  = \frac12 \frac{d}{dr}  \int
(\partial_r R)^2 dt,  \\ 
 s_{5} & =   \int (r^{-1} \partial_t^2 \partial_r R  )(\partial_r \partial_t^{-2} R) dt  =  r^{-1} \int
(\partial_r R)^2 dt.
\end{align*} 

{\bf ii)}  From the mixed terms in $ AR $ we obtain 
\begin{align*}
s_{\rm mixed} & = - 2\varepsilon^2 \int (\partial_t^2 (1+\partial_r^2 + r^{-1}\partial_r) ( A R) )(\partial_r \partial_t^{-2}R) dt \\ 
& = - 2\varepsilon^2 \int ( (1+\partial_r^2 + r^{-1}\partial_r) ( A R) )(\partial_r  R) dt
= s_6 + s_7 + s_8,
\end{align*} 
where 
\begin{align*}
s_6 &:= - 2\varepsilon^2 \int  ( A R) (\partial_rR) dt, \\
s_7 &:= - 2\varepsilon^2 \int  (\partial_r^2( AR)) (\partial_r  R) dt, \\
s_8 &:= - 2\varepsilon^2 \int  (r^{-1}\partial_r( A R) )(\partial_r  R) dt.
\end{align*}
We find 
\begin{equation*}
s_{6} = - \varepsilon^2 \frac{d}{dr} \int   A R^2 dt  
+ \varepsilon^2 \int  ( \partial_r A) R^2  dt,
\end{equation*}
where the second term is estimated by 
\begin{equation*}
\left| \varepsilon^2 \int  ( \partial_r A) R^2  dt \right| \leq \varepsilon^2 \| \partial_r A \|_{L^{\infty}} \| R \|_{L^2}^2
\end{equation*}
which is $ \mathcal{O}(\varepsilon^3) $ since $  \partial_r A = -\varepsilon \partial_{\tau} A + \varepsilon^3 \partial_{\rho} A $ by the chain rule.
Next we have 
\begin{align*}
s_{7} = - 2\varepsilon^2 \int  (\partial_r^2 A)R (\partial_r  R) dt - 3\varepsilon^2 \int  (\partial_r A) (\partial_r  R)^2 dt  - \varepsilon^2 \frac{d}{dr} \int   A (\partial_r  R)^2 dt, 
\end{align*} 
which are estimated by 
\begin{align*}
\left| 2\varepsilon^2 \int  (\partial_r^2 A)R (\partial_r  R) dt \right|  & \leq  2\varepsilon^2 \| \partial_r^2 A \|_{L^{\infty}} \|  R \|_{L^2} \| \partial_r  R \|_{L^2}, \\
\left| 3\varepsilon^2 \int  (\partial_r A) (\partial_r  R)^2 dt \right| & \leq 3\varepsilon^2 \| \partial_r A \|_{L^{\infty}} \| \partial_r  R \|_{L^2}^2. 
\end{align*} 
These terms are at least of order $ \mathcal{O}(\varepsilon^3) $ 
since $  \partial_r A = \mathcal{O}(\varepsilon) $ and $  \partial_r^2 A = \mathcal{O}(\varepsilon^2) $ by the chain rule. For the last term, we obtain the estimate
\begin{equation*}
|s_{8}|  \leq 2\varepsilon^2 r^{-1} \| \partial_r A \|_{L^{\infty}} \|  R \|_{L^2} \| \partial_r  R \|_{L^2} 
+ 2\varepsilon^2 r^{-1} \|  A \|_{L^{\infty}}  \| \partial_r  R \|^2_{L^2} 
\end{equation*} 
which is of order $ \mathcal{O}(\varepsilon^5) $
since $ r \in [r_0 \varepsilon^{-3}, \rho_1 \varepsilon^{-3}] $.\\

{\bf iii)}  From the quadratic terms in $R$ we  obtain 
\begin{align*}
s_{\rm quad} & = - \varepsilon^{\beta} \int (\partial_t^2 (1+\partial_r^2 + r^{-1}\partial_r) ( R^2) )(\partial_r \partial_t^{-2}R) dt \\ 
& = - \varepsilon^{\beta} \int ( (1+\partial_r^2 + r^{-1}\partial_r) ( R^2) )(\partial_r  R) dt
= s_9 + s_{10} + s_{11},
\end{align*} 
where
\begin{align*}
s_{9} &:= - \varepsilon^{\beta} \int  R^2 (\partial_rR) dt =
 -  \frac{1}{3}  \varepsilon^{\beta} \frac{d}{dr} \int   R^3 dt, \\
s_{10} &:= - \varepsilon^{\beta} \int  (\partial_r^2( R^2)) (\partial_r  R) dt  = - \varepsilon^{\beta} \frac{d}{dr} \int   R (\partial_r R)^2 dt - \varepsilon^{\beta} \int   (\partial_r  R)^3 dt, \\
s_{11} &:=  - \varepsilon^{\beta} \int  (r^{-1}\partial_r( R^2) )(\partial_r  R) dt. 
\end{align*} 
The remaining terms can be estimated by
\begin{align*}
\left| \varepsilon^{\beta} \int   (\partial_r  R)^3 dt \right| &  \leq  \varepsilon^{\beta} \| \partial_r R \|_{L^{\infty}} \| \partial_r  R \|_{L^2}^2, \\ 
\left| \varepsilon^{\beta} \int  (r^{-1}\partial_r( R^2) )(\partial_r  R) dt \right| & \leq  2\varepsilon^{\beta} r^{-1}  \|  R \|_{L^{\infty}} \| \partial_r  R \|^2_{L^2} .
\end{align*}

{\bf iv)} For the terms collected in $ N $ we have 
\begin{align*}
s_{N} & = 
\varepsilon^{-\beta} \int ( \partial_t^2 ( 1 + \partial_r^2 + r^{-1}\partial_r) (N(\varepsilon^2 A + \varepsilon^{\beta} R) - N(\varepsilon^2 A)))(\partial_r  \partial_t^{-2} R) dt \\
& = 
\varepsilon^{-\beta} \int ( ( 1 + \partial_r^2 + r^{-1}\partial_r) (N(\varepsilon^2 A + \varepsilon^{\beta} R) - N(\varepsilon^2 A)))(\partial_r  R) dt 
\end{align*} 
Since $ N(v)  $ is analytic in $ v $ we have the representation 
$ 
N(v) = \sum_{n= 3}^{\infty} a_n v^n
$,
with coefficients $ a_n \in \R $, and so 
we find 
$$
\varepsilon^{-\beta} (N(\varepsilon^2 A + \varepsilon^{\beta} R) - N(\varepsilon^2 A)) = \varepsilon^{-\beta} \sum_{n= 3}^{\infty} a_n \sum_{j=1}^n
\left( \begin{array}{c} n \\ j \end{array} \right) (\varepsilon^2 A)^{n-j} ( \varepsilon^{\beta} R)^{j}.
$$ 
such that these terms are at least of order $ \mathcal{O}(\varepsilon^{4}) $ and make no problems for the estimates w.r.t. powers of $ \varepsilon $.  
However, we have to be careful about the regularity of these terms.  As a an example, we look at the terms with most time derivatives,
namely
\begin{equation*}
\int \partial_r^2 (A^{n-j}R^j) (\partial_r  R) dt = s_{12} + s_{13}+ s_{14}  ,
\end{equation*} 
where 
\begin{align*}
s_{12}  & := \int  (\partial_r^2 (A^{n-j})) R^j (\partial_r  R)   dt, \\
s_{13}  & := 2  \int  (\partial_r (A^{n-j})) (\partial_r (R^j)) (\partial_r  R) dt, \\
s_{14}  & :=  \int  A^{n-j}(\partial_r^2 ( R^j)) (\partial_r  R)  dt \\ 
& = j (j-1) \int  A^{n-j} R^{j-2} (\partial_r  R)^3  dt + j  \int  A^{n-j}  R^{j-1} (\partial_r^2  R) (\partial_r  R)  dt.
\end{align*} 
The second derivatives $ \partial_r^2  R $ is controlled 
in terms of $ R $ and $ \partial_r  R $ by means of \eqref{dttResti}.
As a result, there exists a constant $ C_{\rm l} $ and a smooth monotone function
$ C_{\rm n} $ such that for all $ \varepsilon \in (0,1) $ we have
\begin{align*}
|s_{N} | &  \leq   \varepsilon^{4} C_{\rm l}  (\|  R \|_{L^2}^2 +  \| \partial_r  R \|_{L^2}^2) \\ & \quad + \varepsilon^{2+\beta}  C_{\rm n}  (\|  R \|_{L^{\infty}} +  \| \partial_r  R \|_{L^{\infty}})
(\|  R \|_{L^2}^2 +  \| \partial_r  R \|_{L^2}^2).
\end{align*}

{\bf v)} The residual term gives
\begin{equation*}
s_{15} = \varepsilon^{-\beta} \int (\textrm{Res}(\varepsilon^2 A) ) (\partial_r \partial_t^{-2}  R) dt 
= - \varepsilon^{-\beta} \int \partial_t^{-1}(\textrm{Res}(\varepsilon^2 A) ) (\partial_r \partial_t^{-1}  R) dt .
\end{equation*}
It is estimated by 
\begin{equation*}
|s_{15}| \leq C_{\rm res} \varepsilon^{\frac{13}{2}-\beta} \| \partial_r \partial_t^{-1} R \|_{L^2}, 
\end{equation*}
where $C_{\rm res}$ is defined in Lemma \ref{lem-res-anti}.

\begin{remark}
Without the change of variables $ v = u+u^2 $ we would get additionally the following mixed terms 
\begin{equation*}
 -2 \varepsilon^2 \int(\partial_r^2 ( A R)  )(\partial_r \partial_t^{-2}  R) dt - 2\varepsilon^2 \int (r^{-1}\partial_r ( A R)  )(\partial_r \partial_t^{-2} R) dt  
 \end{equation*}
which cannot be written in an obvious manner as sums of a derivative w.r.t. $r$ and higher order terms. Without the change of variables $ v = u+u^2 $ according to Remark \ref{rem33} we cannot estimate $ \partial_t^{-1}(\textrm{Res}(\varepsilon^2 A) ) $ nor the counterpart to $ s_{15} $. This emphasizes again the necessity of the change of variables $ v = u + u^2$ in order to replace \eqref{const3} with \eqref{const5}.
\end{remark}

\subsection{The $ H^1 $-error  estimates}
\label{sec4b}

The energy quantity will be constructed in Section \ref{sec6} based on the derivative formulas for $ s_1 $,  $ s_3 $, $ s_4 $, and other terms.  It will be used for estimating the terms which we were not able to write as derivatives w.r.t. $r$. Since we need estimates for $ \| R \|_{L^{\infty}} $ we will use Sobolev's embedding 
\begin{equation}
\label{Sobolev-emb}
\| f \|_{L^{\infty}}  \leq C  \| f \|_{H^1}, \quad \forall f \in H^1(\R) 
\end{equation}
and hence we have to extend the energy by additional terms involving $\| \partial_t R \|_{L^2}^2$. To do so, we proceed here as in Section \ref{sec4} but now for the $ L^2 $-error estimates of the $ t $-derivatives. 

Based on the product rule
\begin{equation} \label{tame}
 \| \partial_t (uv) \|_{L^2} \leq \| u \|_{L^{\infty}} \| \partial_t v \|_{L^2}
+ \| v \|_{L^{\infty}} \| \partial_t u \|_{L^2} .
\end{equation}
we have the following generalization of the bound \eqref{dttResti} in Lemma \ref{lem-estimate}.

\begin{lemma}
There exist constant $ C_{\rm l} $, $ C_{t,res} $ and a smooth monotone function $ C_{\rm n} $ such that for all $ \varepsilon \in (0,1) $ we have
\begin{align} \label{dttH1esti}
 \| \partial_r^2  R(r,\cdot) \|_{H^1} &  \leq \varepsilon^{\frac{15}{2}-\beta} C_{\rm res} + C_{\rm l} (1 + \varepsilon^2 C_A)  (\|  R(r,\cdot) \|_{H^1} +  \| \partial_r  R(r,\cdot) \|_{H^1}) \\ & \quad + C_{\rm n}(\|  R(r,\cdot) \|_{L^{\infty}} +  \| \partial_r  R(r,\cdot) \|_{L^{\infty}})
\varepsilon^{\beta} (\|  R(r,\cdot) \|_{H^1} +  \| \partial_r  R(r,\cdot) \|_{H^1}), \nonumber
\end{align}
where $C_{\rm res}$ is defined in Lemma \ref{lem-res} and $C_A$ is defined in (\ref{assAsolu}).
\end{lemma}

To get the $H^1$-error estimates, we multiply \eqref{erroreq} by $ \partial_r R $ and then integrate w.r.t. $ t $. We report details of computations as follows. \\

{\bf i)} From the linear terms in $ R $ we obtain   
\begin{align*}
r_1 & = \int (\partial_r^2 R )(\partial_r R) dt  = \frac12 \frac{d}{dr}  \int
(\partial_r R)^2 dt, \\ 
r_2 & = \int ( r^{-1}\partial_r R  )(\partial_r R) dt  =  r^{-1} \int
(\partial_r R)^2 dt , \\ 
r_3 & =  - \int (\partial_t^2 R )(\partial_r R) dt  = \frac12 \frac{d}{dr}  \int
(\partial_t R)^2 dt, \\ 
r_4 & =  -\int  (\partial_t^2 \partial_r^2 R )(\partial_r R) dt  = \frac12 \frac{d}{dr}  \int
(\partial_r \partial_t R)^2 dt,  \\ 
 r_5 & =  - \int (\partial_t^2 r^{-1}\partial_r R  )(\partial_r R) dt  =  r^{-1} \int
(\partial_r \partial_t R)^2 dt.
\end{align*}

{\bf ii)} From the mixed terms in $AR$ we obtain 
\begin{align*}
r_{\rm mixed} & = 2\varepsilon^2 \int (\partial_t^2 (1+\partial_r^2 + r^{-1}\partial_r) ( A R) )(\partial_r R) dt \\ 
& = - 2\varepsilon^2 \int ( (1+\partial_r^2 + r^{-1}\partial_r) \partial_t( A R) )(\partial_r \partial_t R) dt
= r_6 + r_7 + r_8,
\end{align*} 
where 
\begin{align*}
r_6 &:= - 2\varepsilon^2 \int (  \partial_t( A R) )(\partial_r \partial_t R) dt, \\
r_7 &:= - 2\varepsilon^2 \int  (\partial_r^2 \partial_t ( A R) )(\partial_r \partial_t R) dt, \\
r_8 &:= - 2\varepsilon^2 \int (r^{-1}\partial_r \partial_t ( A R)  )(\partial_t \partial_r R) dt.
\end{align*}
We find 
\begin{align*}
r_6 = - 2\varepsilon^2 \int (  \partial_t A ) R (\partial_r \partial_t R) dt - \varepsilon^2 \frac{d}{dr} \int  A (\partial_t R)^2 dt  + \varepsilon^2 \int  (\partial_r A) (\partial_t R)^2 dt,
\end{align*} 
which can be estimated as 
\begin{align*}
\left| 2\varepsilon^2 \int (  \partial_t A ) R (\partial_r \partial_t R) dt\right| & \leq 2\varepsilon^2 \| \partial_t A  \|_{L^{\infty}} \|R \|_{L^2} \| \partial_r \partial_t R \|_{L^2}, \\ 
\left| \varepsilon^2 \int  (\partial_r A) (\partial_t R)^2 dt \right| & \leq \varepsilon^2  \|\partial_r A\|_{L^{\infty}} \|  \partial_t R \|_{L^2}^2.
\end{align*} 
These terms are at least of order $ \mathcal{O}(\varepsilon^3) $ since 
$  \partial_r A $ and $  \partial_t A $ are   of order $ \mathcal{O}(\varepsilon) $ by the chain rule. Next we estimate $r_7$ for which we note that 
\begin{align*}
\frac{d}{dr} \int A (\partial_r \partial_t R)^2 dt = \int (\partial_r A) (\partial_r \partial_t R)^2 dt + 
2 \int A (\partial_r \partial_t R)(\partial_r^2 \partial_t R) dt 
\end{align*} 
and 
\begin{align*}
\partial_r^2 \partial_t ( A R) & = A \partial_r^2 \partial_t R +2 (\partial_rA) \partial_r \partial_t R +( \partial_tA) \partial_r^2  R \\ & \quad +2( \partial_t \partial_r A) \partial_r  R + 
( \partial_r^2 A) \partial_t  R  + (\partial_r^2 \partial_t A) R.
\end{align*} 
As a result, we obtain 
\begin{align*}
r_7  & = - \varepsilon^2 \frac{d}{dr} \int A (\partial_r \partial_t R)^2 dt + r_{7,a} + r_{7,b}+ r_{7,c} + r_{7,d}+ r_{7,e}
\end{align*} 
with 
\begin{align*}
r_{7,a} & := - 3\varepsilon^2 \int  (\partial_r A) (\partial_r \partial_t R)^2 dt , \\
r_{7,b} & := - 2\varepsilon^2 \int  ( \partial_t A) (\partial_r^2  R) (\partial_r \partial_t R) dt , \\
r_{7,c} & := - 4\varepsilon^2 \int  ( \partial_t \partial_r A) (\partial_r  R ) (\partial_r \partial_t R) dt , \\
r_{7,d} & := - 2\varepsilon^2 \int  ( \partial_r^2 A) (\partial_t  R ) (\partial_r \partial_t R) dt , \\
r_{7,e} & := - 2\varepsilon^2 \int  (\partial_r^2 \partial_t A) R (\partial_r \partial_t R) dt .
\end{align*} 
We estimate 
\begin{align*}
|r_{7,a}| & \leq 3\varepsilon^2  \| \partial_rA \|_{L^{\infty}} \| \partial_r \partial_t R\|_{L^2}^2 , \\
|r_{7,b}| &  \leq 2\varepsilon^2 \| \partial_tA\|_{L^{\infty}}\| \partial_r^2  R \|_{L^2} \|\partial_r \partial_t R\|_{L^2} , \\
|r_{7,c}| &  \leq 4\varepsilon^2  \|\partial_t \partial_r A\|_{L^{\infty}}\| \partial_r  R \|_{L^2} \|\partial_r \partial_t R\|_{L^2} , \\
|r_{7,d}| &  \leq 2\varepsilon^2  \| \partial_r^2 A\|_{L^{\infty}}\| \partial_t  R \|_{L^2} \|\partial_r \partial_t R\|_{L^2} , \\
|r_{7,e}| &  \leq 2\varepsilon^2  \|\partial_r^2 \partial_t A\|_{L^{\infty}}\| R \|_{L^2} \|\partial_r \partial_t R\|_{L^2}.
\end{align*} 
All these terms are at least of order $ \mathcal{O}(\varepsilon^3) $ because of the derivatives on $ A $ in $r$ and $t$. Moreover,  we can use  \eqref{dttResti} for estimating $ \| \partial_r^2  R \|_{L^2} $.
The last mixed term is decomposed with the product rule as 
\begin{align*}
r_8 = r_{8,a}+ r_{8,b} + r_{8,c}+ r_{8,d}, 
 \end{align*}
where
 \begin{align*}
r_{8,a} & := - 2\varepsilon^2 \int r^{-1} (\partial_r \partial_t  A) R  (\partial_r \partial_t R) dt,\\ 
 r_{8,b} & := - 2 \varepsilon^2 \int r^{-1}(\partial_t A )(\partial_r  R)(\partial_r \partial_t R) dt, \\
  r_{8,c} & := - 2 \varepsilon^2 \int r^{-1}(\partial_r  A) (\partial_t R) (\partial_r \partial_t R) dt, \\
  r_{8,d} & := - 2\varepsilon^2 \int r^{-1}A (\partial_r \partial_t R)^2 dt.
 \end{align*}
We estimate 
\begin{align*}
|r_{8,a}| &  \leq 2\varepsilon^2  r^{-1} \| \partial_r \partial_t  A\|_{L^{\infty}}\| R  \|_{L^2} \|\partial_r \partial_t R\|_{L^2},\\ 
| r_{8,b}|&  \leq 2 \varepsilon^2  r^{-1}  \|\partial_t A \|_{L^{\infty}}\|\partial_r  R\|_{L^2} \|\partial_r \partial_t R\|_{L^2}, \\
 | r_{8,c}|&  \leq 2 \varepsilon^2 r^{-1} \|\partial_r  A\|_{L^{\infty}}\|\partial_t R\|_{L^2} \|\partial_r \partial_t R\|_{L^2}, \\
 | r_{8,d}|&  \leq 2\varepsilon^2  r^{-1} \| A \|_{L^{\infty}}\|\partial_r \partial_t R\|_{L^2}^2 .
 \end{align*}

{\bf iii)}  From the quadratic terms in $R$ we  obtain 
\begin{align*}
r_{\rm quad} & =  \varepsilon^{\beta} \int (\partial_t^2 (1+\partial_r^2 + r^{-1}\partial_r) ( R^2) )(\partial_r R) dt \\ 
& = - \varepsilon^{\beta} \int ( (1+\partial_r^2 + r^{-1}\partial_r) \partial_t( R^2) )(\partial_r  \partial_tR) dt
= r_9 + r_{10} + r_{11},
\end{align*} 
where 
\begin{align*}
r_{9} &:= - 2\varepsilon^{\beta} \int  R(\partial_tR) (\partial_r \partial_t R) dt, \\
r_{10} &:= - \varepsilon^{\beta} \int (\partial_r^2 \partial_t  (R^2) )(\partial_r  \partial_t  R) dt, \\
r_{11} &:= - \varepsilon^{\beta} \int r^{-1} (\partial_r \partial_t(R^2)  )(\partial_r \partial_t R) dt .
\end{align*} 
The first term is estimated by 
\begin{equation*}
|r_{9} | \leq 2\varepsilon^{\beta} \| R \|_{L^{\infty}} \| \partial_tR\|_{L^2} \| \partial_r \partial_t R\|_{L^2} .
\end{equation*} 
The second term is rewritten by using 
\begin{align*}
 \frac{d}{dr} \int R (\partial_r  \partial_t  R)^2 dt = \int (\partial_r R) (\partial_r  \partial_t  R)^2 dt + 2 \int R (\partial_r  \partial_t  R) (\partial_r^2  \partial_t  R)dt 
 \end{align*}
and 
\begin{equation*}
\partial_r^2 \partial_t  (R^2) =  2 R \partial_r^2 \partial_t R + 4 (\partial_r R) \partial_r \partial_t R
+  2 (\partial_t R) \partial_r^2  R
\end{equation*}
in the form 
$$ 
r_{10} = -  \varepsilon^{\beta}  \frac{d}{dr} \int R (\partial_r  \partial_t  R)^2 dt + r_{10,a} + r_{10,b}
$$ 
with 
\begin{align*}
r_{10,a} & := - 3\varepsilon^{\beta} \int (\partial_r R) (\partial_r  \partial_t  R)^2 dt , \\
r_{10,b} & := -  2\varepsilon^{\beta} \int  (\partial_t R) (\partial_r^2  R) (\partial_r  \partial_t  R) dt.
\end{align*}
The remainder terms are estimated as follows 
\begin{align*}
|r_{10,a} |& \leq  3 \varepsilon^{\beta}  \|\partial_r R\|_{L^{\infty}} \| \partial_r \partial_t R \|_{L^2}^2 , \\
|r_{10,b}|& \leq  2\varepsilon^{\beta}  \|\partial_r^2  R\|_{L^{\infty}} \|\partial_t R\|_{L^2} \|\partial_r  \partial_t  R\|_{L^2} ,
 \end{align*}
where we can use \eqref{dttH1esti} and Sobolev's embedding \eqref{Sobolev-emb} to estimate $\| \partial_r R\|_{L^{\infty}}$ and $\| \partial_r^2 R \|_{L^{\infty}}$. The  last quadratic term is decomposed with the product rule as
\begin{equation*}
r_{11} =  r_{11,a}+r_{11,b},
 \end{equation*}
where
\begin{align*}
 r_{11,a} & :=  - 2 \varepsilon^{\beta} \int r^{-1}R (\partial_r \partial_t R)^2 dt , \\ 
r_{11,b} & := - 2 \varepsilon^{\beta} \int r^{-1} (\partial_r R)(\partial_t R  )(\partial_r \partial_t R)  dt , 
 \end{align*}
which we estimate by 
\begin{align*}
| r_{11,a}|& \leq 2 \varepsilon^{\beta}  r^{-1} \| R \|_{L^{\infty}} \| \partial_r \partial_t R \|_{L^2}^2 , \\ 
|r_{11,b}|  &\leq 2 \varepsilon^{\beta}  r^{-1} \| \partial_r R\|_{L^{\infty}} \| \partial_t R  \|_{L^2} \| \partial_r \partial_t R \|_{L^2}.
 \end{align*}

{\bf iv)} For the terms collected in $ N $ we have 
\begin{align*}
r_{N} & = -\varepsilon^{-\beta} \int ( \partial_t^2 ( 1 + \partial_r^2 + r^{-1}\partial_r) (N(\varepsilon^2 A + \varepsilon^{\beta} R) - N(\varepsilon^2 A)))(\partial_r  R) dt \\
& = 
\varepsilon^{-\beta} \int ( ( 1 + \partial_r^2 + r^{-1}\partial_r) \partial_t(N(\varepsilon^2 A + \varepsilon^{\beta} R) - N(\varepsilon^2 A)))(\partial_r  \partial_t R) dt 
\end{align*} 
Proceeding as for the $ L^2 $-estimate 
and using the bound \eqref{dttH1esti} on the second derivative 
$ \partial_r^2  R $ in terms of $ R $ and $ \partial_r  R $ yields 
the existence of a constant $ C_{14,l} $ and a smooth monotone function $ C_{14,n} $ such that for all $ \varepsilon \in (0,1) $ we have
\begin{align*}
|r_{N} | &  \leq C_{14,l}  \varepsilon^{4}  (\|  R \|_{H^1}^2 +  \| \partial_r  R \|_{H^1}^2) \\ & \quad  + C_{14,n}(\|  R \|_{L^{\infty}} +  \| \partial_r  R \|_{L^{\infty}})
\varepsilon^{2+\beta} (\|  R \|_{H^1}^2 +  \| \partial_r  R \|_{H^1}^2).
\end{align*}

{\bf v)} The residual term
\begin{equation*}
 r_{12} = -\varepsilon^{-\beta} \int (\textrm{Res}(\varepsilon^2 A) )
 (\partial_r R) dt
\end{equation*}
is estimated by 
 \begin{equation*}
|r_{12}| = C_{\rm res} \varepsilon^{\frac{15}{2} -\beta}  \| \partial_r R \|_{L^2},
\end{equation*}
where $C_{\rm res}$ is defined in Lemma \ref{lem-res}.

\subsection{Energy estimates}
\label{sec6}

We use the terms $ s_1 $,  $ s_3 $,  $ s_4 $,  $ r_1 $,  $ r_3 $,  $ r_4 $,  and the parts of $ s_{6} $,  $ s_{7} $, $ s_9 $, $ s_{10}$,  $ r_{6} $,   $ r_{7} $,  and $ r_{10}$ with derivatives in $r$ to define the following energy 
$$ 
E = E_0 + E_1
$$
with 
 \begin{align*}
E_0 & =  \frac12\int
R^2 dt + \frac12 \int
(\partial_r \partial_t^{-1} R)^2 dt + \frac12 \int
(\partial_r  R)^2 dt \\ & \quad  +
\frac12\int
(\partial_t R)^2 dt + \frac12 \int
(\partial_r R)^2 dt + \frac12 \int
(\partial_r \partial_t R)^2 dt , \\
E_1 & =  - \varepsilon^2 \int A   R^2 dt - \varepsilon^2 \int A (\partial_r  R)^2 dt - \frac{1}{3} \varepsilon^{\beta} \int R^3 dt 
-  \varepsilon^{\beta} \int R (\partial_r R)^2 dt
\\ & \quad 
- \varepsilon^2 \int A (\partial_t  R)^2 dt - \varepsilon^2 \int A (\partial_r \partial_t  R)^2 dt
-  \varepsilon^{\beta} \int R (\partial_r \partial_t R)^2 dt.
\end{align*}
The energy part $ E_0 $ is an upper bound for the squared $ H^1 $-norm of $R$, $\partial_t^{-1} R$, and $\partial_r R$. Moreover,  for all $ M > 0 $ there exists an $ \varepsilon_1 > 0 $ such that for all $  \varepsilon 
\in (0,\varepsilon_1) $ we have 
$$ 
\frac12 E_0 \leq E_1 \leq \frac32 E_0
$$ 
as long as $ E^{1/2} \leq M $. All other linear terms which are not contained in the energy $ E $ 
have either  a $ r^{-1} = \varepsilon^3 \rho^{-1} $  in front,  namely 
$ s_2 $, $ s_5$, $ s_8 $,  $ r_2 $, $ r_5$, and $ r_8 $, or contain a time or space derivative of $ A $,  as parts of $s_6$, $s_7$, $r_6$, and $r_7$, and so all other linear terms are at least of  order $ \mathcal{O}(\varepsilon^3) $.
All nonlinear terms have at least a  $ \varepsilon^{4} $ or $ \varepsilon^{\beta} $ in front. The residual terms $ s_{15} $ and $ r_{16} $ are of  order $ \mathcal{O}(\varepsilon^3) $ if $\beta$ is chosen as 
$\beta = \frac{7}{2}$. As a result, we estimate the rate of change of 
energy $E$ from the following inequality 
\begin{align} 
\frac{d}{d r} E &\leq  C \varepsilon^3 E + C \varepsilon^{7/2}      E^{3/2} + C  \varepsilon^3 E^{1/2}  \notag \\
&\leq  2 C \varepsilon^3 E + C \varepsilon^{7/2}      E^{3/2} + C  \varepsilon^3,
\label{smitseq4}
\end{align}
with a constant $ C $ independent of $   \varepsilon \in (0,\varepsilon_1 ) $ as lomg as $ E^{1/2} \leq M $.
Under the assumption that $ C \varepsilon^{1/2}      E^{1/2}  \leq 1 $ we obtain 
$$ 
\frac{d}{d t} E \leq (2 C+1) \varepsilon^3 E + C  \varepsilon^3.
$$
Gronwall's inequality immediately gives the bound 
$$ 
\sup_{t \in [0,T_0/\varepsilon^3]} E(t) = C T_0 e^{ (2 C+1)T_0} =: M = \mathcal{O}(1)
$$ 
and so $ \sup\limits_{t \in [0,T_0/\varepsilon^3]} \|  R(t) \|_{H^1} = \mathcal{O}(1) $.
Finally choosing $ \varepsilon_2 > 0 $ so small that $ C \varepsilon_2^{1/2}      M^{1/2}  \leq 1 $ gives the required estimate for all $  \varepsilon \in (0,\varepsilon_0 ) $
with  $ \varepsilon_0 = \min ( \varepsilon_1 , \varepsilon_2 )> 0$. Therefore, we have proved Theorem \ref{mainthapp}.

\section{Solitary wave solutions of the cKdV equation}
\label{newsec3}

Here we prove Theorem \ref{th-soliton}. We look for solutions of the cKdV equation (\ref{cKdV-intro}) in the class of solitary waves represented in the form
\begin{equation}
\label{Hirota-u}
A(\rho,\tau) = -6  \partial_{\tau}^2 \log f(\rho,\tau),
\end{equation}
which tranforms \eqref{cKdV-intro} to the following bilinear equation \cite{NakamuraChen80}:
\begin{equation}
\label{Hirota-cKdV}
2 \left[  f \partial_{\rho} \partial_{\tau} f - (\partial_{\rho} f) (\partial_{\tau} f) \right] + \rho^{-1} f \partial_{\tau} f + f \partial_{\tau}^4 f - 4 (\partial_{\tau} f) (\partial_{\tau}^3 f) + 3 (\partial_{\tau}^2 f)^2 = 0.
\end{equation}
To prove Theorem \ref{th-soliton}, we analyze solutions of \eqref{Hirota-cKdV} in the self-similar form  \cite{Step1,Step2,Step3}:
\begin{equation}
\label{self-similar}
f(\rho,\tau) = 1 + \frac{1}{(6 \rho)^{1/3}} F(z), \qquad z = \frac{\tau}{(6 \rho)^{1/3}}
\end{equation}
with some $F \in C^{\infty}(\R,\R)$. The form (\ref{Hirota-u}) and (\ref{self-similar}) yields (\ref{A-Hirota}). We give a complete characterization for all possible solutions for $F(z)$ and prove that there exist no square integrable function $A(\rho,\tau)$ w.r.t. $\tau$. The proof is based on the three results obtained in the following three lemmas.

The first result gives the most general expression for $F(z)$ in \eqref{self-similar}. 

\begin{lemma}
	\label{lemma-1}
	The most general solution $f(\rho,\tau)$ of the bilinear equation (\ref{Hirota-cKdV}) in the self-similar form \eqref{self-similar} with $F \in C^{\infty}(\R,\R)$ is given by
	\begin{equation}
	\label{lin-super}
	F(z) = \alpha \left[ (w_1')^2 - z w_1^2 \right] \pm 2 \sqrt{\alpha \beta} 
	\left[ w_1' w_2' - z w_1 w_2 \right] + \beta \left[ (w_2')^2 - z w_2^2 \right],
	\end{equation}
	where $\alpha,\beta \in \mathbb{R}$ are arbitrary such that $\alpha \beta \geq 0$ and $w_1(z) := {\rm Ai}(z)$, $w_2(z) := {\rm Bi}(z)$  are two linearly independent solutions of the Airy equation 
	\begin{equation}
	\label{Airy}
	w''(z) - z w(z) = 0.
	\end{equation}
\end{lemma}

\begin{proof}
	Substituting \eqref{self-similar} into \eqref{Hirota-cKdV} shows that the variables are separated and $F(z)$ satisfies an overdetermined system of two (linear and quadratic) differential equations:
	\begin{equation}
	\label{ode-1}
	F''''(z) - 4 z F''(z) - 2 F'(z) = 0
	\end{equation}
	and
	\begin{equation}
		\label{ode-2}
4 F'(z) [ z F'(z) + F(z) - F'''(z) ] + 3 [F''(z)]^2 = 0.
\end{equation}	
Let $G(z) := -F'(z)$. Then \eqref{ode-1} reduces to the third-order equation 
$$
G'''(z) - 4 z G'(z) - 2 G(z) = 0,
$$
the general solution of which is known (see 10.4.57 in \cite{AS1972}):
	\begin{equation}
\label{G-expansion}
G(z) = \alpha [{\rm Ai}(z)]^2 + \beta [{\rm Bi}(z)]^2 + \gamma {\rm Ai}(z) {\rm Bi}(z),
\end{equation}
where $\alpha,\beta,\gamma$ are arbitrary. Denoting $w_1(z) := {\rm Ai}(z)$ and $w_2(z) := {\rm Bi}(z)$, we confirm that 
\begin{align*}
\frac{d}{dz} [(w_{1,2}')^2 - z w_{1,2}^2] &= 2 w_{1,2}' (w_{1,2}'' - z w_{1,2}) - w_{1,2}^2 = -w_{1,2}^2
\end{align*}
and 
\begin{align*}
\frac{d}{dz} [ w_1' w_2' - z w_1 w_2] &= (w_1''-zw_1) w_2' + w_1'(w_2''-z w_2) - w_1 w_2 = -w_1 w_2
\end{align*}
Hence, $F'(z) = -G(z)$ is integrated to the form
	\begin{equation}
\label{F-expression}
F(z) = C + \alpha \left[ (w_1')^2 - z w_1^2 \right] + \beta \left[ (w_2')^2 - z w_2^2 \right] + \gamma 
\left[ w_1' w_2' - z w_1 w_2 \right],
\end{equation}
where $C$ is an integration constant. The same constant $C$ appears in the integration of \eqref{ode-1} to the form 
	\begin{equation}
\label{ode-3}
F'''(z) - 4 z F'(z) + 2 F(z) = 2 C.
\end{equation}
It remains to verify if the general solution \eqref{F-expression} satisfies the quadratic equation \eqref{ode-2}.  Multiplying \eqref{ode-3} by $F''(z)$ and integrating, we obtain 
\begin{equation}
\label{ode-4}
[F''(z)]^2 - 4 z [F'(z)]^2 + 4 F(z) F'(z) = 4C F'(z) + D,
\end{equation}
where $D$ is another integration constant. On the other hand, substituting \eqref{ode-3} into \eqref{ode-2} yields 
\begin{equation}
\label{ode-5}
[F''(z)]^2 - 4 z [F'(z)]^2 + 4 F(z) F'(z) = \frac{8}{3} C F'(z).
\end{equation}
Comparison of \eqref{ode-4} and \eqref{ode-5} yields $C = D = 0$. Finally, we substitute \eqref{F-expression} into \eqref{ode-5} with $C = 0$ and obtain
\begin{align*}
0 &= [F''(z)]^2 - 4z [F'(z)]^2 + 4 F(z) F'(z) \\
&= (\gamma^2 - 4 \alpha \beta) (w_1 w_2' - w_1' w_2)^2,
\end{align*}
where the Wronskian of two linearly independent solutions is nonzero, 
$w_1 w_2' - w_1' w_2 \neq 0$. Hence, the system \eqref{ode-1}-\eqref{ode-2} is compatible for the solution \eqref{F-expression} if and only if $C = 0$ and $\gamma = \pm 2 \sqrt{\alpha \beta}$ with only two arbitrary constants $\alpha, \beta \in \mathbb{R}$. 
\end{proof}

The solution $F(z)$ in \eqref{lin-super} is real if and only $\alpha \beta \geq 0$. 
The next result shows that the expression \eqref{self-similar} with this $F$ is sign-definite (positive) if and only if $\alpha \geq 0$ and $\beta = 0$. 

\begin{lemma}
	\label{lemma-2}
Let $F$ be given by \eqref{lin-super} with $\alpha \beta \geq 0$. For every $k > 0$, we have $k + F(z) > 0$ for every $z \in \mathbb{R}$ if and only if $\alpha \geq 0$ and $\beta = 0$. 
\end{lemma}

\begin{proof}
	We shall make use the asymptotic expansion of the Airy functions, see 10.4.59-60 and 10.4.63-64 in \cite{AS1972}:
	\begin{align*}
	\left\{ \begin{array}{l}
	{\rm Ai}(z) \sim \displaystyle \frac{1}{2 \sqrt{\pi} \sqrt[4]{z}} e^{-\frac{2}{3} z^{3/2}} \left[ 1 + \mathcal{O}(z^{-3/2}) \right], \\	
	{\rm Bi}(z) \sim \displaystyle \frac{1}{\sqrt{\pi} \sqrt[4]{z}} e^{\frac{2}{3} z^{3/2}} \left[ 1 + \mathcal{O}(z^{-3/2}) \right], 
	\end{array} \right. \qquad \mbox{\rm as} \;\; z \to +\infty
	\end{align*}
	and
	\begin{align*}
	\left\{ \begin{array}{l}
	{\rm Ai}(z) \sim \displaystyle \frac{1}{\sqrt{\pi} \sqrt[4]{|z|}}  \left[ \sin\left( \frac{2}{3} |z|^{3/2} + \frac{\pi}{4}\right) + \mathcal{O}(|z|^{-3/2}) \right], \\	
	{\rm Bi}(z) \sim \displaystyle \frac{1}{\sqrt{\pi} \sqrt[4]{|z|}}  \left[ \cos\left( \frac{2}{3} |z|^{3/2} + \frac{\pi}{4}\right) + \mathcal{O}(|z|^{-3/2}) \right],
	\end{array} \right. \qquad \mbox{\rm as} \;\; z \to -\infty
	\end{align*}
Due to cancelations, it is not convenient to use the expression \eqref{lin-super} directly as $z \to \pm \infty$. Instead, we use \eqref{G-expansion} with $\gamma = \pm 2 \sqrt{\alpha \beta}$ and obtain 
\begin{align*}
F'(z) &\sim - \frac{\alpha}{4 \pi \sqrt{z}} e^{-\frac{4}{3} z^{3/2}} \left[ 1 + \mathcal{O}(z^{-3/2}) \right] - \frac{\beta}{\pi \sqrt{z}} e^{\frac{4}{3} z^{3/2}} \left[ 1 + \mathcal{O}(z^{-3/2}) \right] \\ 
& \quad \mp \frac{\sqrt{\alpha \beta}}{\pi \sqrt{z}}\left[ 1 + \mathcal{O}(z^{-3/2}) \right] \qquad \mbox{\rm as} \;\; z \to +\infty
\end{align*}
and 
\begin{align*}
F'(z) &\sim - \frac{\alpha}{2 \pi \sqrt{|z|}} \left[ 1 + \sin\left( \frac{4}{3} |z|^{3/2} \right) + \mathcal{O}(|z|^{-3/2}) \right] \\
& \quad - \frac{\beta}{2 \pi \sqrt{|z|}} \left[ 1 - \sin\left( \frac{4}{3} |z|^{3/2} \right) + \mathcal{O}(|z|^{-3/2}) \right] \\ 
& \quad \pm \frac{\sqrt{\alpha \beta}}{\pi \sqrt{|z|}}\left[ \cos\left( \frac{4}{3} |z|^{3/2} \right) + \mathcal{O}(|z|^{-3/2}) \right] \qquad \mbox{\rm as} \;\; z \to -\infty
\end{align*}
Integrating these expressions and recalling that $C = 0$ in \eqref{F-expression}, we obtain 
\begin{align*}
F(z) &\sim \frac{\alpha}{8 \pi z} e^{-\frac{4}{3} z^{3/2}} \left[ 1 + \mathcal{O}(z^{-3/2}) \right] - \frac{\beta}{2 \pi z} e^{\frac{4}{3} z^{3/2}} \left[ 1 + \mathcal{O}(z^{-3/2}) \right] \\ 
& \quad \mp \frac{2 \sqrt{\alpha \beta}}{\pi} \sqrt{z} \left[ 1 + \mathcal{O}(z^{-3/2}) \right] \qquad \mbox{\rm as} \;\; z \to +\infty
\end{align*}
and 
\begin{align*}
F(z) &\sim \frac{\alpha}{\pi} \sqrt{|z|} \left[ 1 + \mathcal{O}(|z|^{-3/2}) \right] + \frac{\beta}{\pi} \sqrt{|z|} \left[ 1 + \mathcal{O}(|z|^{-3/2}) \right] \\ 
& \quad \mp \frac{\sqrt{\alpha \beta}}{2 \pi |z|}\left[ \sin\left( \frac{4}{3} |z|^{3/2} \right) + \mathcal{O}(|z|^{-3/2}) \right] \qquad \mbox{\rm as} \;\; z \to -\infty
\end{align*}
If $\beta \neq 0$, then $F(z) \to -{\rm sgn}(\beta) \infty$ as $z \to +\infty$. 
Since $\alpha \beta \geq 0$, we also get $F(z) \to {\rm sgn}(\beta) \infty$ as $z \to -\infty$. Hence for every $k \geq 0$, $k + F(z)$ is not sign-definite 
for every $\beta \neq 0$. 

Setting $\beta = 0$, we get $F'(z) = - \alpha [{\rm Ai}(z)]^2$ and since ${\rm Ai}(z) \to 0$ as $z \to +\infty$ sufficiently fast, we can define 
\begin{equation}
\label{expression-F}
F(z) = \alpha \int_z^{\infty} [{\rm Ai}(z')]^2 dz',
\end{equation}
where the constant of integration is uniquely selected since $C = 0$ in \eqref{F-expression}. Hence, $F(z)$ is sign-definite for every $z \in \mathbb{R}$ and ${\rm sgn}(F) = {\rm sgn}(\alpha)$. We also have $F(z) \to 0$ as $z \to +\infty$ and $F(z) \to {\rm sgn}(\alpha) \infty$ as $z \to -\infty$. Hence, for every $k > 0$, $k + F(z) > 0$ for every $z \in \mathbb{R}$ if and only if $\alpha \geq 0$ in \eqref{expression-F}.
\end{proof}

Finally, we use the solution $F(z)$ in \eqref{expression-F} with $\alpha > 0$ and show that the solution $A(\rho,\cdot)$ in \eqref{Hirota-u} and (\ref{self-similar}) decay to zero at infinity, satisfies the zero-mean constraint, but is not square integrable for every $\rho > 0$.

\begin{lemma}
	\label{lemma-3}
	Let $F$ be given by \eqref{expression-F} with $\alpha > 0$ 
	and let $A$ be given by \eqref{Hirota-u} with (\ref{self-similar}). For every $\rho > 0$, we have $A(\rho,\tau) \to 0$ as $|\tau| \to \infty$,
	$
	\int_{\mathbb{R}} A(\rho,\tau) d \tau = 0,
	$
	and $A(\rho,\cdot) \notin L^2(\mathbb{R})$.
\end{lemma}

\begin{proof}
By chain rule, we have from \eqref{Hirota-u} and \eqref{self-similar} 
$$
A(\rho,\tau) = -\frac{6}{(6 \rho)^{2/3}} \partial_z^2 \log [(6 \rho)^{1/3} + F(z)],
$$
where $z = \tau/(6 \rho)^{1/3}$. Since $k + F(z) > 0$ for every $k > 0$ and $z \in \mathbb{R}$, we have 
$A(\rho,\cdot) \in L^2_{\rm loc}(\mathbb{R})$. It remains to consider square integrability of $A(\rho,\cdot)$ at infinity. 

It follows from \eqref{expression-F}, see the proof of Lemma \ref{lemma-2}, that 
\begin{align*}
F(z) \sim \frac{\alpha}{8 \pi z} e^{-\frac{4}{3} z^{3/2}} \left[ 1 + \mathcal{O}(z^{-3/2}) \right] \qquad \mbox{\rm as} \;\; z \to +\infty
\end{align*}
and 
\begin{align*}
F(z) \sim \frac{\alpha}{\pi} \sqrt{|z|} \left[ 1 + \mathcal{O}(|z|^{-3/2}) \right] \qquad \mbox{\rm as} \;\; z \to -\infty.
\end{align*}
Since $F(z), F'(z) \to 0$ as $z \to +\infty$, we have 
\begin{align}
A(\rho,\tau) &\sim - \frac{6}{\rho} F''(z) \left[ 1 + \mathcal{O}(|z|^{-3/2}) \right] \nonumber \\
&\sim 
-\frac{\alpha}{2 \pi \rho} e^{-\frac{4}{3} z^{3/2}}  \left[ 1 + \mathcal{O}(|z|^{-3/2}) \right] \qquad \mbox{\rm as} \;\; z \to +\infty,
\label{decay-1}
\end{align}
hence, $A(\rho,\cdot) \in L^2(\tau_0,\infty)$ for any $\tau_0 \gg 1$ and $\rho > 0$. However, since $F(z) \to \infty$ and $F'(z) \to 0$ as $z \to -\infty$, we have 
\begin{align}
A(\rho,\tau) &\sim -\frac{6}{(6 \rho)^{2/3}} \frac{F''(z)}{(6 \rho)^{1/3} + F(z)} \left[ 1 + \mathcal{O}(|z|^{-3/2}) \right] \nonumber \\
&\sim -\frac{\sqrt{6}}{\sqrt{\rho |\tau|}} \left[ \cos\left( \frac{4}{3} |z|^{3/2}\right)
+ \mathcal{O}(|z|^{-3/2}) \right]
\qquad \mbox{\rm as} \;\; z \to -\infty,
\label{decay-2}
\end{align}	
where we have used the expansion
\begin{align*}
F''(z) \sim \frac{\alpha}{\pi} \left[ \cos\left( \frac{4}{3} |z|^{3/2} \right) + \mathcal{O}(|z|^{-3/2}) \right] \qquad \mbox{\rm as} \;\; z \to -\infty.
\end{align*}
Hence, $A(\rho,\cdot) \notin L^2(-\infty,\tau_0)$ for any $\tau_0 \ll -1$ and $\rho > 0$. At the same time, $A(\rho,\tau) \to 0$ as $\tau \to \pm \infty$ 
and the zero-mean constraint is satisfied due to 
$$
\int_{\mathbb{R}} A(\rho,\tau) d \tau = -\frac{6}{(6 \rho)^{1/3}} \frac{F'(z)}{(6 \rho)^{1/3} + F(z)} \biggr|_{z \to -\infty}^{z \to +\infty} = 0,
$$
due to the decay of $F'(z) \to 0$ as $z \to \pm \infty$.
\end{proof}

Figure \ref{fig-1} shows a representative example of the solitary wave in the cKdV equation \eqref{cKdV-intro}, where $A$ is plotted versus $\tau$ for four values of $\rho = 1, 20, 100, 500$. The oscillatory tail behind the solitary wave ruins localization of the solitary wave in $L^2(\mathbb{R})$. Similar to \cite{Step1,Step2}, we use very large value of $\alpha$ to detach the solitary wave from the oscillatory tail. For larger values of $\rho$, the solitary wave departs even further from the oscillatory tail but its amplitude also decays to zero.

\begin{figure}[htb!]
	\centering
	\includegraphics[width=7cm,height=6cm]{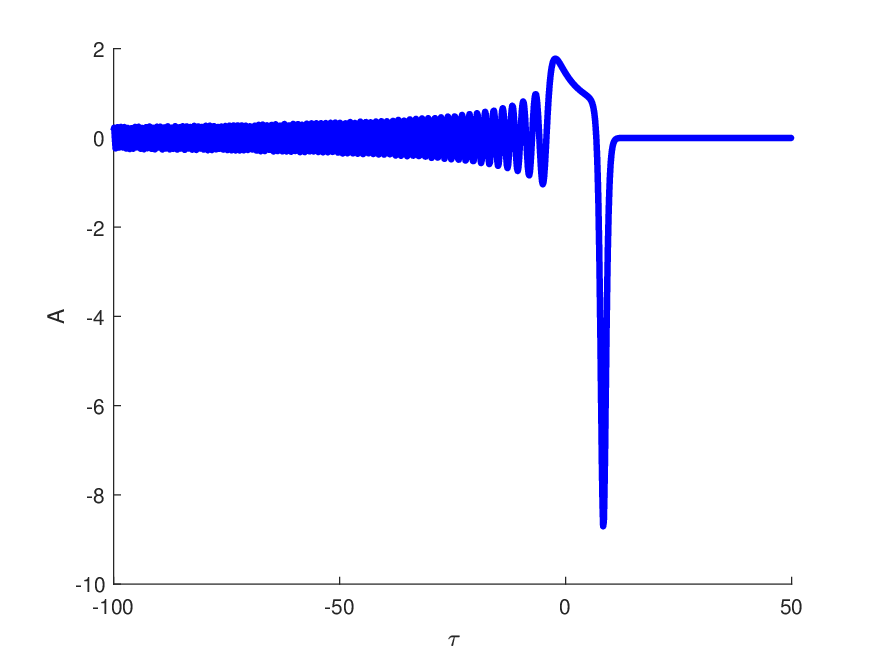}
	\includegraphics[width=7cm,height=6cm]{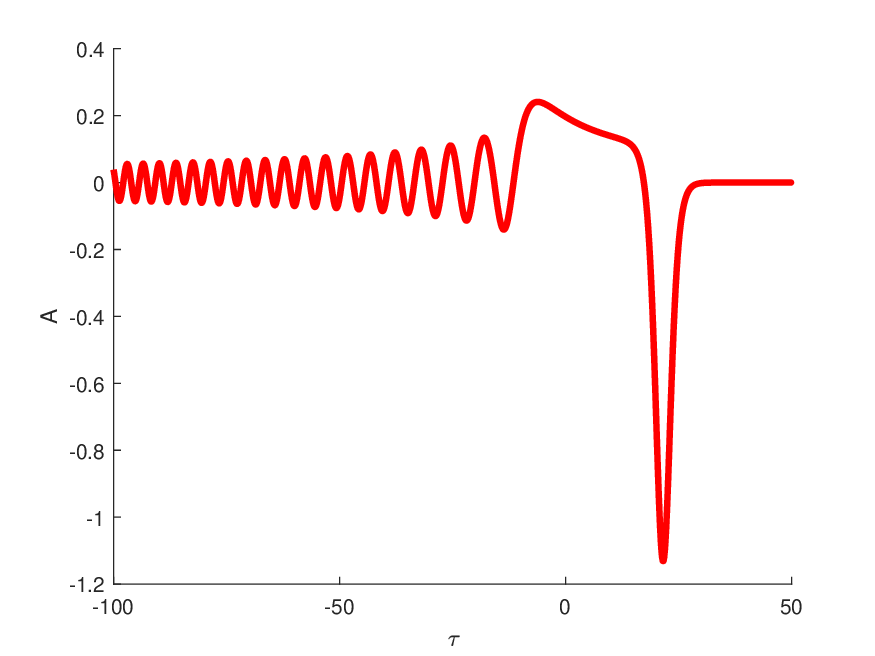} \\
	\includegraphics[width=7cm,height=6cm]{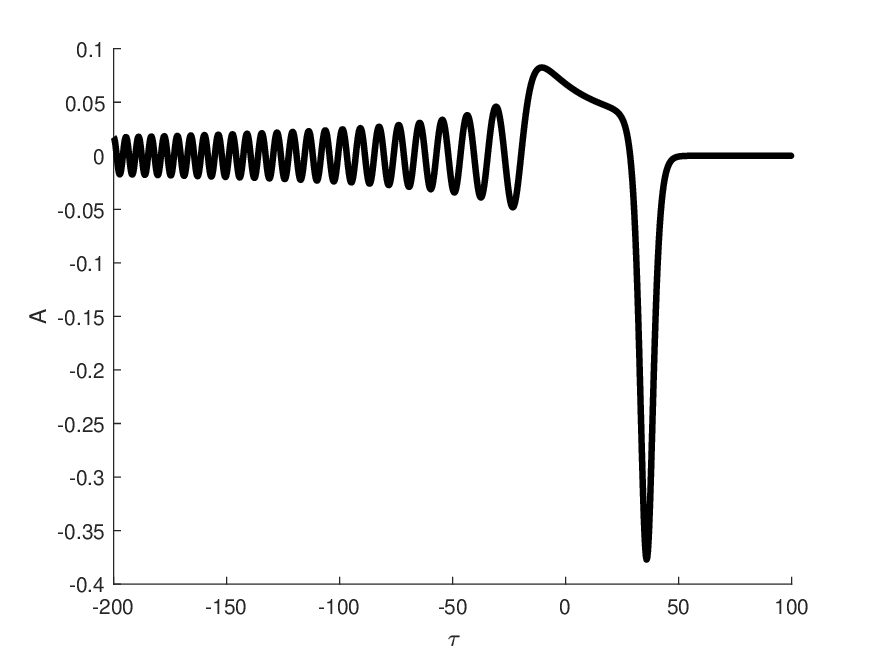}
	\includegraphics[width=7cm,height=6cm]{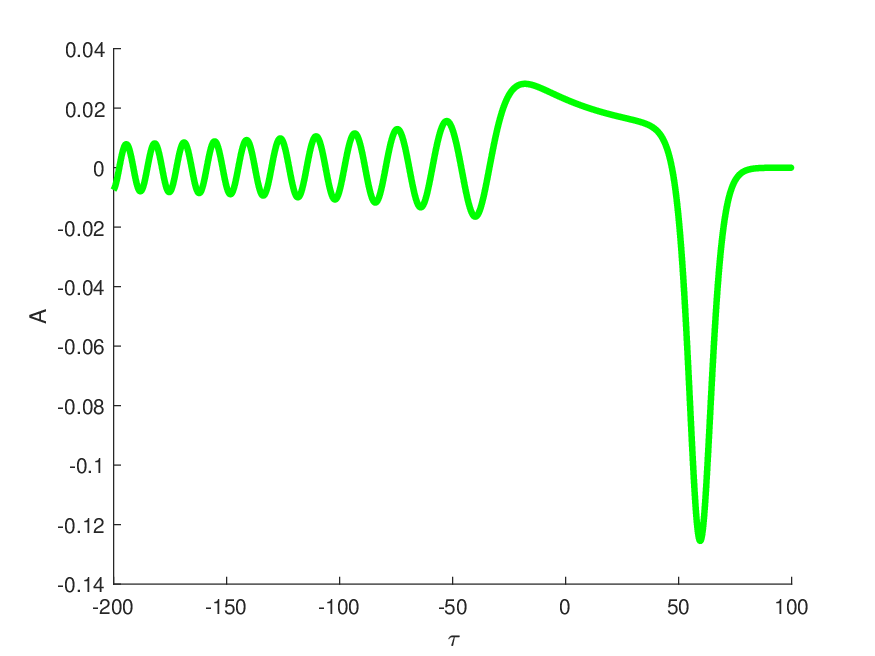}	
	\caption{The soliton solution in the form \eqref{Hirota-u} with \eqref{self-similar} and \eqref{expression-F} for $\alpha = 10^8$ versus $\tau$ for $\rho = 1$ (top left), $\rho = 20$ (top right), $\rho = 100$ (bottom left), and $\rho = 500$ (bottom right). }
	\label{fig-1}
\end{figure}

\section{Discussion}
\label{newsec4}

We have addressed here the justification of the cKdV equation 
(\ref{cKdV-intro}) in the context of the radial waves diverging from the origin in the 2D regularized Boussinesq equation (\ref{constintro}).
We have shown that the spatial dynamics and temporal dynamics formulations of (\ref{constintro}) are not well posed simultaneously. If the temporal dynamics formulation is well posed, the spatial dynamics formulation is ill posed and vice versa. We have justified the cKdV equation (\ref{cKdV-intro}) in the case of the spatial dynamics formulation (\ref{Bous-rad})--(\ref{ivp-radial}). The main result of Theorem \ref{mainthapp} relies on the existence of smooth solutions of the cKdV equation (\ref{cKdV-intro}) with the zero-mean constraint (\ref{mean-value}) in the class of functions (\ref{class-solutions}) with Sobolev exponent $s > \frac{17}{2}$. However, we have also showed in Theorem \ref{th-soliton} that the class of solitary wave solutions decaying at infinity satisfies the zero-mean constraint but fails to 
be square integrable due to the oscillatory, weakly decaying tail as $\tau \to -\infty$.

This work calls for further study of the applicability of the cKdV equation for the radial waves in nonlinear dispersive systems. We will list several open directions. 

\begin{figure}[htb!]
	\centering
	\includegraphics[width=7cm,height=6cm]{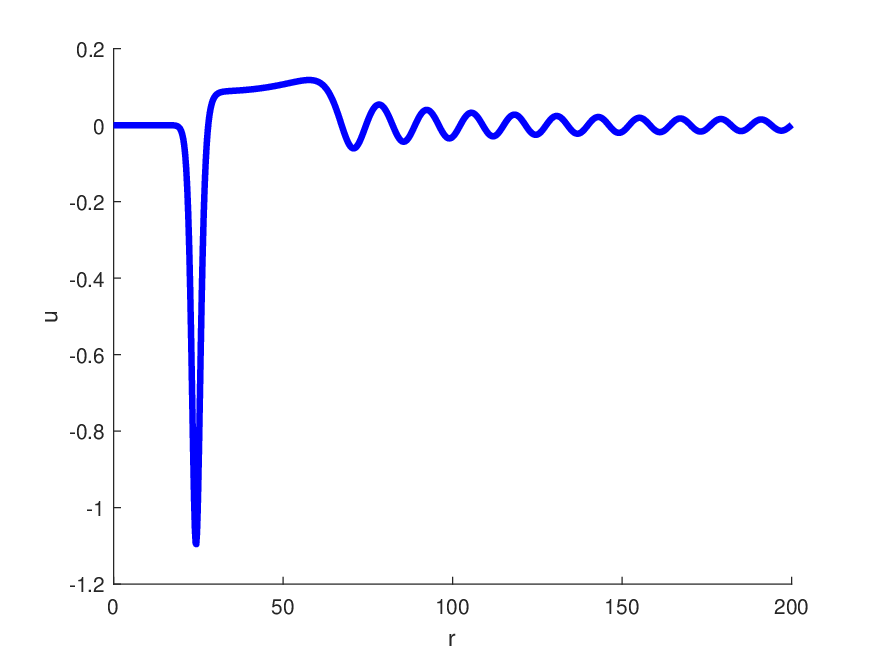}
	\includegraphics[width=7cm,height=6cm]{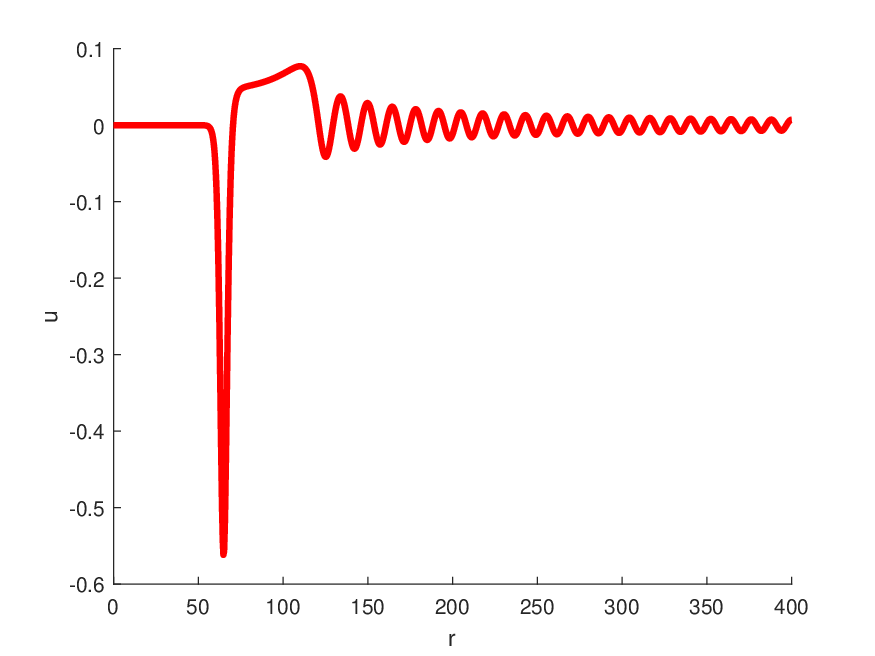}	
	\caption{The soliton solution in the form \eqref{soliton-time} 
		with \eqref{expression-F} for $\alpha = 10^8$ and $\eps = 0.1$ versus $t$ for $t = 50$ (left) and $t = 100$ (right). }
	\label{fig-2}
\end{figure}

First, the solitary waves of the cKdV equation (\ref{cKdV-intro}) can be written as the approximate solutions of the radial Boussinesq equation 
(\ref{Bous-rad}) in the form:
\begin{equation}
\label{soliton-time}
u(r,t) = -\frac{6}{(6r)^{2/3}} \left(  \frac{F''\left(\frac{t-r}{(6r)^{1/3}}\right)}{(6 r)^{1/3} \eps + F\left(\frac{t-r}{(6r)^{1/3}}\right)} - \frac{\left[ F'\left(\frac{t-r}{(6r)^{1/3}}\right) \right]^2}{\left[ (6 r)^{1/3} \eps + F\left(\frac{t-r}{(6r)^{1/3}}\right) \right]^2} \right),
\end{equation}
where $F(z)$ is given by (\ref{expression-F}) with $\alpha > 0$ and $\varepsilon > 0$ is the small parameter of asymptotic expansions. These solitary waves can be considered for fixed $t > 0$ as functions of $r$ on $(0,\infty)$, see Figure \ref{fig-2} for $\eps = 0.1$. The solitary waves decay very fast as $r \to 0$ and decay as $\mathcal{O}(r^{-1})$ as $r \to \infty$, see \eqref{decay-1} and \eqref{decay-2}. However, they are still not square integrable in the radial variable because $\int_0^{\infty} r u(r,t)^2  dr$ diverges for every $t > 0$. In addition, the cKdV equation (\ref{cKdV-intro}) is ill-posed as the temporal 
dynamics formulation from $t = 0$ to $t > 0$.

Second, it might be possible to consider the temporal formulation 
of the cKdV equation (\ref{cKdV-intro}) and to justify 
it in the framework of the temporal dynamics formulation of the Boussinesq equation (\ref{constintro}) with $\sigma = -1$. One needs to construct a stable manifold for the cKdV equation (\ref{cKdV-intro}) and to prove the error estimates on the stable manifold. The stable part of the linear semigroup for the cKdV equation (\ref{cKdV-intro}) has a decay rate of $  t^{-3} $ for $ t \to \infty $ due to $ \lambda = -|k|^{1/3} $, which could be sufficient for the construction of the stable manifold. However, one needs to combine the linear estimates with the nonlinear estimates.

Third, one can consider a well-posed 2D Boussinesq equation (\ref{constintro}) with $ \sigma = -1 $ and to handle the ill-posed  radial spatial dynamics formulation (\ref{Bous-rad})--(\ref{ivp-radial}) with the justification of the cKdV approximation as in Theorem \ref{mainthapp} by using the approach from \cite{KanoNishida,SchnICIAM}. This would involve working in spaces of functions which are analytic in a strip in the complex plane.
The oscillatory tails of the cKdV approximation, see Figure \ref{fig-2}, would now accumulate towards $r \to 0$ for the well-posed 2D Boussinesq equation, see Figure 4 in \cite{Step2}, with the rate of $\mathcal{O}(r^{-1/2})$ as $r \to 0$ which is sufficient for $\int_0^{\infty} r u(r,t)^2  dt$ to converge for every $t > 0$. 

We conclude that the most promising problem for future work is to justify the temporal formulation of the cKdV equation (\ref{cKdV-intro}) for the temporal formulation of the 2D Boussinesq equation (\ref{constintro}) with $\sigma = -1$, for which the solitary waves are admissible in the $L^2$-based function spaces. If this justification problem can be solved, one can then consider 
the transverse stability problem of cylindrical solitary waves under the azimuthal perturbations within the approximation given by the cKP equation with the exact solutions found in \cite{Step2,Step3}.

\end{document}